%% file: main_arXiv.tex
\documentclass[11pt,a4paper]{article}
\pdfoutput=1

\input{packages}
\input{macros}        
\input{bibsetup}

\title{\LARGE \bfseries Sum-of-squares chordal decomposition of polynomial matrix inequalities}

\author[1]{Yang Zheng\thanks{
    \href{mailto:zhengy@eng.ucsd.edu}{zhengy@eng.ucsd.edu}}}
\author[2]{Giovanni Fantuzzi\thanks{\href{mailto:giovanni.fantuzzi10@imperial.ac.uk}{giovanni.fantuzzi10@imperial.ac.uk}}}
\affil[1]{Department of Electrical and Computer Engineering, University of California San Diego, CA, 92093, US.}
\affil[2]{Department of Aeronautics, Imperial College London, London, SW7 2AZ, UK.}

\begin{document}

\maketitle
\begin{abstract}
\noindent
We prove decomposition theorems for sparse positive (semi)definite polynomial matrices that can be viewed as sparsity-exploiting versions of the Hilbert--Artin, Reznick, Putinar, and Putinar--Vasilescu Positivstellens\"atze.
First, we establish that a polynomial matrix $P(x)$ with chordal sparsity is positive semidefinite for all $x\in \mathbb{R}^n$ if and only if there exists a sum-of-squares (SOS) polynomial $\sigma(x)$ such that $\sigma P$ is a sum of sparse SOS matrices. 
Second, we show that setting $\sigma(x)=(x_1^2 + \cdots + x_n^2)^\nu$ for some integer $\nu$ suffices if $P$ is homogeneous and positive definite globally.
Third, we prove that if $P$ is positive definite on a compact semialgebraic set $\mathcal{K}=\{x:g_1(x)\geq 0,\ldots,g_m(x)\geq 0\}$ satisfying the Archimedean condition, then $P(x) = S_0(x) + g_1(x)S_1(x) + \cdots + g_m(x)S_m(x)$ for matrices $S_i(x)$ that are sums of sparse SOS matrices. 
Finally, if $\mathcal{K}$ is not compact or does not satisfy the Archimedean condition, we obtain a similar decomposition for $(x_1^2 + \ldots + x_n^2)^\nu P(x)$ with some integer $\nu\geq 0$ when $P$ and $g_1,\ldots,g_m$ are homogeneous of even degree. Using these results, we find sparse SOS representation theorems for polynomials that are quadratic and correlatively sparse in a subset of variables, and we construct new convergent hierarchies of sparsity-exploiting SOS reformulations for convex optimization problems with large and sparse polynomial matrix inequalities.
Numerical examples demonstrate that these hierarchies can have a significantly lower computational complexity than traditional ones.
\end{abstract}

\begin{small}\noindent
\end{small}


\section{Introduction}
\label{s:intro}
\input{ss-introduction}

\section{Chordal decomposition of polynomial matrices}
\label{s:main-results}
\input{ss-main-results}

\section{Convex optimization with sparse polynomial matrix inequalities}
\label{ss:applications-to-PMI-optimiz}
\input{ss-convex-pmi}

\section{Relation to correlatively sparse SOS decompositions of polynomials}
\label{section:scalarcase}
\input{ss-csp-relation}

\section{Numerical experiments}
\label{s:examples}
\input{ss-examples-numerical}

\section{Proofs}
\label{s:proofs}
\input{ss-proof-nonexistence}
\input{ss-proof-hilbert-artin}
\input{ss-proof-putinar}
\input{ss-proof-putinar-vasilescu}

\section{Conclusion}
\label{Section:Conclusion}
\input{ss-conclusion}

\vspace*{4ex}
\noindent
\textbf{Acknowledgements.}
We would like to thank Antonis Papachristodoulou, Pablo Parrilo, J. William Helton, Igor Klep and Licio Romao for insightful conversations that have led to this work. We also thank the reviewers and Associate Editor, who motivated us to prove stronger theorems than those included in our original manuscript. Their suggestions considerably improved the quality of this work.

\appendix
\input{aa-psd-not-sos-example-proof}
\input{aa-homogenization}

\input{aa-polynomial-diagonalization}
\input{aa-weighted-sos-matrix-lemma}

\pdfbookmark{References}{bookmark:references}
\bibliography{references.bib}

\end{document}

%% file: packages.tex
\usepackage[latin1]{inputenc}
\usepackage[english]{babel}
\usepackage[left=3cm, right=3cm, top=3cm, bottom=3cm]{geometry}

\usepackage{authblk}

\usepackage[titletoc,title,page]{appendix}

\usepackage{footnote}

\usepackage{amsmath}
\usepackage{amsfonts}
\usepackage{amssymb}
\usepackage{amsthm}
\usepackage{bbm}		
\usepackage{mathtools}	
\usepackage{bbm}
\usepackage{xcolor}
\usepackage{tikz}

\usepackage{graphicx}
\usepackage{subfigure}
\usepackage[suffix=]{epstopdf}
\usepackage{booktabs}
\usepackage[font=small,labelfont=bf]{caption} 
\usepackage{multirow}
\usepackage{multicol}
\usepackage{array}
\newcolumntype{C}[1]{>{\centering\arraybackslash}p{#1}}

\definecolor{thelinkcolor}{RGB}{0,0,150}
\usepackage{hyperref}

\makeatletter
    \def\cl@chapter{\@elt {theorem}}
\makeatother

\usepackage[nameinlink]{cleveref}
\crefname{equation}{}{}
\crefname{theorem}{Theorem}{Theorems}
\crefname{corollary}{Corollary}{Corollaries}
\crefname{example}{Example}{Examples}
\crefname{lemma}{Lemma}{Lemmas}
\crefname{proposition}{Proposition}{Propositions}
\crefname{figure}{Figure}{Figures}
\crefname{table}{Table}{Tables}
\crefname{section}{Section}{Sections}
\crefname{appendix}{Appendix}{Appendices}
\Crefname{equation}{}{}
\Crefname{theorem}{Theorem}{Theorems}
\Crefname{corollary}{Corollary}{Corollaries}
\Crefname{example}{Example}{Examples}
\Crefname{lemma}{Lemma}{Lemma}
\Crefname{proposition}{Proposition}{Proposition}
\Crefname{figure}{Figure}{Figures}
\Crefname{table}{Table}{Tables}
\Crefname{section}{Section}{Sections}
\Crefname{appendix}{Appendix}{Appendices}

\hypersetup{colorlinks=true,
	urlcolor=blue,
	linkcolor=blue,
	citecolor=blue,
	bookmarksdepth=paragraph,
	pdftitle={Chordal sparsity, matrix decomposition, and sos optimization},
	pdfauthor={Y. Zheng and G. Fantuzzi},}

\usepackage{tikz}
\usetikzlibrary{shapes,matrix,decorations.shapes}
\usetikzlibrary{arrows,decorations.pathmorphing,backgrounds,positioning,fit,matrix}
\usetikzlibrary{decorations.pathreplacing}
\usetikzlibrary{shapes,arrows,chains}
\usetikzlibrary[calc]
\usetikzlibrary[plotmarks]
\usetikzlibrary[patterns,decorations]

%% file: macros.tex
\usepackage[shortlabels]{enumitem}
\setlist[itemize]{itemindent=0ex,itemsep=-0.5ex,leftmargin=2ex,topsep=5pt}
\setlist[enumerate]{label={\arabic*)},leftmargin=*,nolistsep,noitemsep,topsep=1ex}

\newcommand{\abs}[1]{\left\vert #1 \right\vert}
\newcommand{\tr}{{{\mathsf T}}}

\newcommand\xqed[1]{\leavevmode\unskip\penalty9999 \hbox{}\nobreak\hfill\quad\hbox{#1}}
\newcommand\markendexample{\xqed{{\small$\blacksquare$}}}

\definecolor{fwFillBlue}{rgb}{0.803921580314636   0.878431379795074   0.968627452850342}
\definecolor{fwFillRed}{rgb}{0.992156863212585   0.917647063732147   0.796078443527222}
\definecolor{fwFillGreen}{rgb}{0.839215695858002   0.909803926944733   0.850980401039124}
\definecolor{fwFillMagenta}{rgb}{0.937254905700684   0.866666674613953   0.866666674613953}
\definecolor{fwLineBlue}{rgb}{0.000000000000000   0.447058826684952   0.741176486015320}
\definecolor{fwLineRed}{rgb}{0.850980401039124   0.325490206480026   0.098039217293262}
\definecolor{fwLineGreen}{rgb}{0.000000000000000   0.498039215803146   0.000000000000000}
\definecolor{fwLineMagenta}{rgb}{0.494117647409439   0.184313729405403   0.556862771511078}
\definecolor{mygreen}{RGB}{77,175,74}
\definecolor{myred}{RGB}{228,26,28}
\definecolor{myblue}{RGB}{55,126,184}
\definecolor{matlabgray}{RGB}{127,127,127}
\definecolor{matlabblue}{RGB}{0,113,188}
\definecolor{matlabred}{RGB}{216,82,24}
\definecolor{matlabgreen}{rgb}{0.4660,0.6740,0.1880}
\definecolor{matlabcyan}{rgb}{0.3010,0.7450,0.9330}   
\definecolor{matlabyellow}{rgb}{0.9290,0.6940,0.1250}
\definecolor{matlaborange}{RGB}{255,153,102}
\definecolor{matlabpurple}{rgb}{0.4940,0.1840,0.5560}
\definecolor{matlabsafered}{RGB}{215,25,28}
\definecolor{matlabsafegreen}{RGB}{171,221,164}
\definecolor{lightbrown}{RGB}{149,99,99}
\definecolor{darkbrown}{RGB}{82,48,48}
\newcommand\solidrule[1][10pt]{\rule[0.5ex]{#1}{1.5pt}}

\DeclareMathOperator{\Diag}{Diag}							

\newtheorem{theorem}{Theorem}
\newtheorem{lemma}{Lemma}
\newtheorem{proposition}{Proposition}

\newtheorem{corollary}{Corollary}
\theoremstyle{definition}
\newtheorem{example}{Example}
\newtheorem{remark}{Remark}

\numberwithin{equation}{section}
\numberwithin{theorem}{section}
\numberwithin{corollary}{section}
\numberwithin{lemma}{section}
\numberwithin{proposition}{section}
\numberwithin{definition}{section}
\numberwithin{remark}{section}
\numberwithin{example}{section}

\makeatletter
\newcommand{\subalign}[1]{%
	\vcenter{%
		\Let@ \restore@math@cr \default@tag
		\baselineskip\fontdimen10 \scriptfont\tw@
		\advance\baselineskip\fontdimen12 \scriptfont\tw@
		\lineskip\thr@@\fontdimen8 \scriptfont\thr@@
		\lineskiplimit\lineskip
		\ialign{\hfil$\m@th\scriptstyle##$&$\m@th\scriptstyle{}##$\crcr
			#1\crcr
		}%
	}
}
\newcommand{\pushright}[1]{\ifmeasuring@#1\else\omit\hfill$\displaystyle#1$\fi\ignorespaces}
\newcommand*\fsize{\dimexpr\f@size pt\relax}
\makeatother

\newcommand{\mysquare}[2]{%
	\protect\begin{tikzpicture}%
	\protect\draw[thick,color=#1,fill=#1] (0,0) -- (#2,0) -- (#2,#2) -- (0,#2) -- (0,0);
	\protect\end{tikzpicture}%
}

%% file: bibsetup.tex
\usepackage[sort&compress,numbers]{natbib}

%% file: ss-introduction.tex
Many control problems for systems of ordinary differential equations can be posed as convex optimization problems with matrix inequality constraints that must hold on a prescribed portion of the state space~\cite{chesi2010lmi,lasserre2010moments,henrion2011inner,scherer2006lmi}. For differential equations with polynomial right-hand side, these problems often take the generic form
%
\begin{equation}\label{e:infinite-sdp-intro}
B^* := \inf_{\lambda \in \mathbb{R}^\ell} \quad 
b(\lambda) \quad
\text{s.t.} \quad P(x,\lambda) := P_0(x) - \sum_{i=1}^\ell P_i(x)\lambda_i \succeq 0 \quad\forall x \in \mathcal{K},
\end{equation}
where 
$b:\mathbb{R}^\ell \to \mathbb{R}$ is a convex
cost function,
$P_0,\ldots,P_\ell$ are $m \times m$ symmetric polynomial matrices depending on the system state $x \in \mathbb{R}^n$, and
\begin{equation}
\label{e:semialgebgraic-set}
\mathcal{K} = \left\{ x \in \mathbb{R}^n:\; g_1(x)\geq 0,\, \ldots,\, g_q(x) \geq 0 \right\}
\end{equation}
is a basic semialgebraic set defined by inequalities on fixed polynomials $g_1,\,\ldots,\,g_q$. There is no loss of generality in considering only inequality constraints because any equality $g(x)=0$ can be replaced by the two inequalities $g(x)\geq 0$ and $-g(x)\geq 0$.

Verifying polynomial matrix inequalities is generally an NP-hard problem~\cite{Murty1987}, which makes~\cref{e:infinite-sdp-intro} intractable. Nevertheless, feasible vectors $\lambda$ can be found via semidefinite programming if one imposes the stronger condition that
\begin{equation}\label{e:matrix-sos}
P(x,\lambda) = S_0(x) + g_1(x)S_1(x) + \cdots + g_q(x) S_q(x)
\end{equation}
for some $m \times m$ sum-of-squares (SOS) polynomial matrices $S_0,\,\ldots,\,S_q$. A polynomial matrix $S(x)$ is SOS if $S(x)=H(x)^\tr H(x)$ for some polynomial matrix $H(x)$, and it is well known~\cite{gatermann2004symmetry,kojima2003sums,parrilo2013semidefinite,scherer2006matrix} that linear optimization problems with SOS matrix variables can be reformulated as semidefinite programs (SDPs). However, the size of these SDPs increases very rapidly as a function of the size of $P$, its polynomial degree, and the number of independent variables $x$. Thus, even though in theory SDPs can be solved using algorithms with polynomial-time complexity~\cite{boyd2004convex,nemirovski2006advances,nesterov1994interior,vandenberghe1996semidefinite}, in practice reformulations of~\cref{e:infinite-sdp-intro} based on~\cref{e:matrix-sos} remain intractable because they require prohibitively large computational resources.

This work introduces new sparsity-exploiting SOS decompositions that can be used to efficiently certify the nonnegativity of large but sparse polynomial matrices, where ``sparse" means that many of their off-diagonal entries are identically zero. Specifically, let $P(x)$ be an $m \times m$ polynomial matrix and describe its sparsity using an undirected graph $\mathcal{G}$ with vertices $\mathcal{V}=\{1,\ldots,m\}$ and edges $\mathcal{E} \subseteq \mathcal{V} \times \mathcal{V}$ such that $P_{ij}(x)=P_{ji}(x)\equiv0$ when $i \neq j$ and $(i,j) \notin \mathcal{E}$. Motivated by chordal decomposition techniques for semidefinite programming~\cite{fukuda2001exploiting,nakata2003exploiting,sun2014decomposition,vandenberghe2015chordal,zheng2020chordal}, we ask whether the computational complexity of~\cref{e:matrix-sos} can be lowered by decomposing the matrices $S_0,\ldots,S_q$ into sums of sparse SOS matrices, with nonzero entries only on the principal submatrix indexed by one of the maximal cliques of the sparsity graph $\mathcal{G}$ of $P$. We prove that this clique-based decomposition exists if $\mathcal{G}$ is a chordal graph (meaning that, for every cycle of length larger than three, there is at least one edge in $\mathcal{E}$ connecting nonconsecutive vertices in the cycle), $\mathcal{K}$ is a compact set satisfying the so-called Archimedean condition, and $P(x)$ is strictly positive definite on $\mathcal{K}$ (cf. \cref{th:sparse-putinar}). This result is a sparsity-exploiting version of Putinar's Positivstellensatz~\cite{putinar1993positive} for polynomial matrices. We also give a sparse-matrix version of the Putinar--Vasilescu Positivstellensatz~\cite{PutinarVasilescu1999}, stating that $(x_1^2 + \cdots + x_n^2)^\nu P$ admits a clique-based SOS decomposition for some integer $\nu \geq 0$ if $P$ is homogeneous, has even degree, and is positive definite on a semialgebraic set $\mathcal{K}$ defined by homogeneous polynomials $g_1, \ldots, g_m$ of even degree (cf. \cref{th:sparse-putinar-vasilescu-homog}). This result applies even if $\mathcal{K}$ is noncompact. For the particular case of global nonnegativity, $\mathcal{K}\equiv \mathbb{R}^n$, we immediately recover a sparse-matrix version of Reznick'z Positivestellensatz~\cite{Reznick1995} (cf. \cref{th:sparse-reznick-homog}), and further prove a version of the Hilbert--Artin theorem~\cite{artin1927zerlegung} where the strict positivity of $P$ is weakened into positive semidefiniteness upon replacing the factor $(x_1^2 + \cdots + x_n^2)^\nu$ with a generic SOS polynomial (cf. \cref{th:weighted-chordal-decomposition}). \Cref{table:summary-results} summarizes our results and gives references to their counterparts for polynomials and general (dense) polynomial matrices.

\begin{table}
    \centering
    \caption{Summary of Positivstellens\"atze for polynomials, polynomial matrices, and polynomial matrices with structural sparsity.}
    \label{table:summary-results}
    \begin{tabular}{c|ccc}
    \toprule
        \multirow{ 2}{*}{Positivstellenatz} & \multirow{ 2}{*}{Polynomials}  & {General polynomial} & {Sparse polynomial}\\
         &   & {matrices} & {matrices}\\
        \midrule
        Hilbert--Artin & Artin~\cite{artin1927zerlegung} & Du~\cite{Du2017} & \cref{th:weighted-chordal-decomposition}\\ 
        Reznick & Reznick~\cite{Reznick1995} & Dinh \textit{et al.}~\cite{Dinh2021} & \cref{th:sparse-reznick-homog}\\ 
        Putinar & Putinar~\cite{putinar1993positive} & Scherer \& Hol~\cite{scherer2006matrix} & \cref{th:sparse-putinar}\\ 
        Putinar--Vasilescu & Putinar \& Vasilescu~\cite{PutinarVasilescu1999} & Dinh \textit{et al.}~\cite{Dinh2021} & \cref{th:sparse-putinar-vasilescu-homog}\\ 
        \bottomrule
    \end{tabular}
\end{table}

These chordal SOS decomposition theorems for polynomial matrices extend a classical chordal decomposition result for constant (i.e., independent of $x$) positive semidefinite (PSD) sparse matrices~\cite{agler1988positive}. The latter allows for significant computational gains when applied to large-scale sparse SDPs~\cite{sun2014decomposition,zheng2020chordal}, analysis and control of structured systems~\cite{andersen2014robust,Zheng2017Scalable}, and optimal power flow for large grids~\cite{andersen2014reduced,molzahn2013implementation}. Similarly, our decomposition results can be used to construct convergent hierarchies of sparsity-exploiting SOS reformulations of problem~\cref{e:infinite-sdp-intro} (cf. \cref{th:sos-program-convergence_putinar,th:sos-program-convergence-homog,th:sos-program-convergence-inhomog}), which produce a minimizing sequence of feasible vectors $\lambda$ and often have a significantly lower computational complexity compared to traditional approaches based on the ``dense'' weighted SOS representation~\cref{e:matrix-sos}. 


Finally, when the polynomial matrix $P$ in~\cref{e:infinite-sdp-intro} is not only sparse, but also depends only on a small set of $n$-variate monomials, our chordal SOS decompositions can be combined with known methods to exploit \textit{term sparsity}. These methods include facial reduction~\cite{reznick1978extremal,permenter2014basis,lofberg2009pre}, symmetry reduction~\cite{gatermann2004symmetry,Riener2013}, the exploitation of so-called correlative sparsity in the couplings between the independent variables~\cite{waki2006sums,lasserre2006convergent,grimm2007note,klep2019sparse,josz2018lasserre}, and the recent TSSOS, chordal-TSSOS and CS-TSSOS approaches to polynomial optimization~\cite{Wang2019term-sparsity,Wang2020chordal-tssos,Wang2020tssos,Wang2020cs-tssos}. Even though all of these methods have been developed for polynomial inequalities, rather than polynomial matrix inequalities, they can be applied directly upon reformulating the matrix inequality $P(x; \lambda) \succeq 0$ on $\mathcal{K}$ as the polynomial inequality $p(x,y)= y^\tr P(x; \lambda) y \geq0$ for all $x \in \mathcal{K}$ and $y \in \mathbb{R}^m$ with $\|y\|_\infty \leq 1$. In particular, if $P$ is structurally sparse, then $p(x,y)$ is correlatively term sparse with respect to $y$, and the techniques of~\cite{waki2006sums,lasserre2006convergent,grimm2007note,zheng2018sparse,Wang2020chordal-tssos,josz2018lasserre} can be used to check if it is nonnegative for all $x$ and $y$ of interest. This connection does not make our matrix decomposition theorems redundant: on the contrary, they reveal that correlatively sparse SOS decompositions for $p(x,y)$  depend only \textit{quadratically} on $y$ (\cref{corollary:globalcase,,corollary:global_even,corollary:compact}), which cannot be concluded from the available SOS decomposition theorems for scalar polynomials.

The rest of this work is structured as follows. 
\Cref{s:main-results} states our main chordal SOS decomposition results, while \cref{ss:applications-to-PMI-optimiz} explains how they can be used to formulate convergent hierarchies of sparsity-exploiting SOS reformulations of problem~\cref{e:infinite-sdp-intro}. 
\Cref{section:scalarcase} relates our decomposition results for polynomial matrices to the classical SOS techniques for correlatively sparse polynomials~\cite{waki2006sums,lasserre2006convergent,grimm2007note}.
Computational examples are presented in \cref{s:examples}. Our matrix decomposition results are proven in \cref{s:proofs}, and conclusions are offered in \cref{Section:Conclusion}. Appendices contain details of calculations and proofs of auxiliary results. 

%% file: ss-main-results.tex
The main contributions of this work are chordal decomposition theorems for $n$-variate PSD polynomial matrices $P(x)$ whose sparsity is described by a chordal graph $\mathcal{G}$. After reviewing the connection between sparse matrices and graphs, as well as the standard chordal decomposition theorem for constant matrices, we present decomposition theorems that apply globally (\cref{ss:results-global}) and on basic semialgebraic sets (\cref{ss:results-semialgebraic}).

\subsection{Sparse matrices and chordal graphs}\label{s:preliminaries}

A graph $\mathcal{G}$ is a set of vertices $\mathcal{V}=\{1,\dots, m\}$ connected by a set of edges $\mathcal{E} \subseteq \mathcal{V} \times \mathcal{V}$. We call $\mathcal{G}$ \textit{undirected} if edge $(j,i)$ is identified with edge $(i,j)$, so edges are unordered pairs; \textit{complete} if $\mathcal{E} = \mathcal{V}\times \mathcal{V}$; \textit{connected} if there exists a path $(i,v_1),\,(v_1,v_2),\,\ldots,\,(v_k,j)$ between any two distinct vertices $i$ and $j$. We consider only undirected graphs, and focus mainly on the connected but not complete case.

A vertex $i \in \mathcal{V}$ of an undirected graph is called \textit{simplicial} if the subgraph induced by its neighbours is complete. A subset of vertices $\mathcal{C} \subseteq \mathcal{V}$ that are fully connected, meaning that $(i,j) \in \mathcal{E}$ for all pairs of (distinct) vertices $i,j \in \mathcal{C}$, is called a \textit{clique}.  A clique is \textit{maximal} if it is not contained in any other clique. Finally, a sequence of vertices $\{v_1, v_2, \ldots, v_k\} \subseteq \mathcal{V}$ with $k \geq 3$ is called a \textit{cycle} of length $k$ if $(v_i, v_{i+1}) \in \mathcal{E}$ for all $i = 1, \ldots, k-1$ and $(v_k, v_{1}) \in \mathcal{E}$. Any edge $(v_i,v_j)$ between nonconsecutive vertices in a cycle is known as a \textit{chord}, and a graph is said to be \textit{chordal} if all cycles of length $k\geq 4$ have at least one chord. Complete graphs, chain graphs, and trees are all chordal; other particular examples are illustrated in \cref{F:ChordalGraph}. Any non-chordal graph can be made chordal by adding appropriate edges to it; the process is known as a \textit{chordal extension}~\cite{vandenberghe2015chordal}.

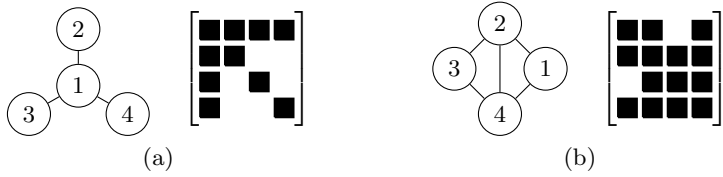
\begin{figure}[t]
	\centering
	\footnotesize
	\subfigure[]{
		\begin{tikzpicture}
		\node[draw, circle] (1) at (0,0) {\footnotesize 1};
		\node[draw, circle] (2) at (90:0.75cm) {\footnotesize 2};
		\node[draw, circle] (3) at (210:0.75cm) {\footnotesize 3};
		\node[draw, circle] (4) at (330:0.75cm) {\footnotesize 4};
		\draw (1)--(2);
		\draw (1)--(3);
		\draw (1)--(4);
		\end{tikzpicture}
		\begin{tikzpicture}
		\node at (0,0) {
			\setlength\arraycolsep{0.75pt}
			\def\arraystretch{0.6}
			$
			\begin{bmatrix}\\[-2.4pt]
			\mysquare{black}{7pt} & \mysquare{black}{7pt} & \mysquare{black}{7pt} & \mysquare{black}{7pt} \\    
			\mysquare{black}{7pt} & \mysquare{black}{7pt} & &\\    
			\mysquare{black}{7pt} && \mysquare{black}{7pt} &\\
			\mysquare{black}{7pt} &&&\mysquare{black}{7pt}\\[1.6pt]
			\end{bmatrix}
			$};
		\end{tikzpicture}
	}
	\hspace{30pt}
	\subfigure[]{
		\begin{tikzpicture}
		\node[draw, circle] (1) at (0:0.6cm) {\footnotesize 1};
		\node[draw, circle] (2) at (90:0.6cm) {\footnotesize 2};
		\node[draw, circle] (3) at (180:0.6cm) {\footnotesize 3};
		\node[draw, circle] (4) at (270:0.6cm) {\footnotesize 4};
		\draw (1)--(2);
		\draw (2)--(3);
		\draw (3)--(4);
		\draw (4)--(1);
		\draw (2)--(4);
		\end{tikzpicture}
		\begin{tikzpicture}
		\node at (0,0) {
			\setlength\arraycolsep{0.75pt}
			\def\arraystretch{0.6}
			$
			\begin{bmatrix}\\[-2.4pt]
			\mysquare{black}{7pt} & \mysquare{black}{7pt} &  & \mysquare{black}{7pt} \\    
			\mysquare{black}{7pt} & \mysquare{black}{7pt} & \mysquare{black}{7pt}&\mysquare{black}{7pt}\\    
			 &\mysquare{black}{7pt}& \mysquare{black}{7pt} & \mysquare{black}{7pt}\\
			\mysquare{black}{7pt} &\mysquare{black}{7pt}&\mysquare{black}{7pt} &\mysquare{black}{7pt}\\[1.6pt]
			\end{bmatrix}
			$};
		\end{tikzpicture}	
	}
	\caption{Connected, non-complete, chordal undirected graphs, and the sparse matrices they describe.
		(a)~Star graph with vertices $\mathcal{V}=\{1,2,3,4\}$ and edges $\mathcal{E}=\{(1,2),(1,3),(1,4)\}$. 
		(b)~Triangulated graph with vertices $\mathcal{V}=\{1,2,3,4\}$ and edges $\mathcal{E}=\{(1,2),(2,3),(3,4),(1,4),(2,4)\}$.}
	\label{F:ChordalGraph}
\end{figure}

The sparsity pattern of any $m \times m$ symmetric matrix $P$ can be described using an undirected graph $\mathcal{G}$  with vertices $\mathcal{V}=\{1,\ldots,m\}$ and an edge set $\mathcal{E}$ such that $(i,j)\notin \mathcal{E}$ if and only if $i\neq j$ and $P_{ij}=0$; see \cref{F:ChordalGraph} for two examples. We call $\mathcal{G}$ the sparsity graph of $P$. Dense principal submatrices of $P$ are indexed by cliques of $\mathcal{G}$, and maximal dense principal submatrices are indexed by maximal cliques.

For each maximal clique $\mathcal{C}_k$ of $\mathcal{G}$, define a matrix $E_{\mathcal{C}_k} \in \mathbb{R}^{|\mathcal{C}_k| \times m}$ as
\begin{equation} \label{E:IndexMatrix}
(E_{\mathcal{C}_k})_{ij} := 
\begin{cases}
1, &\text{if } \mathcal{C}_k(i) = j, \\ 
0, &\text{otherwise},
\end{cases}
\end{equation}
where $|\mathcal{C}_k|$ is the cardinality of $\mathcal{C}_k$ and $\mathcal{C}_k(i)$ is the $i$-th vertex in $\mathcal{C}_k$. 
This definition ensures that the operation $E_{\mathcal{C}_k}^\tr X_k E_{\mathcal{C}_k}$ ``inflates" a $\vert\mathcal{C}_k\vert \times \vert \mathcal{C}_k\vert $ matrix $X_k$ into a sparse $m\times m$ matrix with nonzero entries only in the submatrix indexed by $\mathcal{C}_k$; for example, if $m=3$, $\mathcal{C}_k=\{1,3\}$, and $S = \left[\begin{smallmatrix}\alpha & \beta\\ \beta & \gamma\end{smallmatrix}\right]$ we have
\begin{equation*}
E_{\mathcal{C}_k} = \begin{bmatrix} 1 & 0 & 0 \\
0 & 0 &1 \end{bmatrix} \qquad \text{and} \qquad 
E_{\mathcal{C}_k}^\tr S E_{\mathcal{C}_k} = \begin{bmatrix}
\alpha & 0 & \beta\\
0 & 0 & 0\\
\beta & 0 & \gamma
\end{bmatrix}.
\end{equation*}
The following classical result states that PSD matrices with a chordal sparsity graph admit a clique-based PSD decomposition.

\begin{theorem}[Agler \textit{et al.}~\cite{agler1988positive}]\label{th:ChordalDecompositionTheorem}
	A matrix $P$ whose sparsity graph is chordal and has maximal cliques $\mathcal{C}_1,\ldots,\mathcal{C}_t$ is positive semidefinite if and only if there exist positive semidefinite matrices 
	$S_k$ of size $\abs{\mathcal{C}_k}\times\abs{\mathcal{C}_k}$ 
	such that $$P = \sum_{k=1}^{t} E_{\mathcal{C}_k}^\tr S_k E_{\mathcal{C}_k}.$$
\end{theorem}

\begin{figure}
	\centering
	\footnotesize
	\begin{tikzpicture}
	\matrix (m) [matrix of nodes,row sep = 0em,	column sep = 2em, nodes={circle, draw=black}] at (0,0) {1 & 2 & 3\\};
	\draw (m-1-1) -- (m-1-2);
	\draw (m-1-2) -- (m-1-3);
	\end{tikzpicture}
	\caption{Chordal sparsity graph of the $3 \times 3$ matrices used in \cref{ex:3by3maxtirx,ex:nondecomposable-3d-matrix,ex:nondecomposable-multivariate,ex:sparse-putinar}. 
	}
	\label{F:LineGraph}
\end{figure}
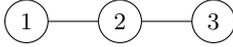

\begin{example}\label{ex:3by3maxtirx}
    The PSD matrix
    $P = 
    \left[\begin{smallmatrix} 
    4 & 2 & 0 \\
    2 & 2 & 2 \\
    0 & 2 & 4
    \end{smallmatrix}\right]
    $
    has the sparsity graph illustrated in \cref{F:LineGraph}, which is chordal because it has no cycles. This graph has maximal cliques $\mathcal{C}_1 = \{1,2\}$ and $\mathcal{C}_2 = \{2,3\}$. The decomposition guaranteed by \cref{th:ChordalDecompositionTheorem} reads $P = E_{\mathcal{C}_1}^\tr S_1 E_{\mathcal{C}_1} + E_{\mathcal{C}_2}^\tr S_2 E_{\mathcal{C}_2}$ with 
    $S_1 = 
    \left[\begin{smallmatrix} 
    4 & 2\\
    2 & 1
    \end{smallmatrix}\right]
    $
    and
    $S_2 = 
    \left[\begin{smallmatrix} 
    1 & 2 \\
    2 & 4
    \end{smallmatrix}\right]
    $. \markendexample
\end{example}

Our goal is to derive versions of \cref{th:ChordalDecompositionTheorem} for sparse polynomial matrices that are positive semidefinite, either globally or on a basic semialgebraic set, where the matrices $S_k$ are polynomial and SOS. This allows us to build convergent hierarchies of sparsity-exploiting SOS reformulations for the optimization problem~\cref{e:infinite-sdp-intro}, which have a considerably lower computational complexity compared to standard (dense) ones.  
Throughout the paper, we assume without loss of generality that the sparsity graph $\mathcal{G}$ of $P(x)$ is connected and not complete. Complete sparsity graphs correspond to dense matrices, while disconnected ones correspond to matrices that have a block-diagonalizing permutation. Each irreducible diagonal block can be analyzed individually and has a connected (but possibly complete) sparsity graph by construction.

\subsection{Polynomial matrix decomposition on \texorpdfstring{$\mathbb{R}^n$}{Rn}}
\label{ss:results-global}

Let the polynomial matrix $P(x)$ be positive semidefinite for all $x \in \mathbb{R}^n$ and have a chordal sparsity graph with maximal cliques $\mathcal{C}_1, \ldots, \mathcal{C}_t$. Applying \cref{th:ChordalDecompositionTheorem} for each $x\in \mathbb{R}^n$ yields PSD matrices $S_1(x),\,\ldots,\,S_t(x)$ such that
\begin{equation}\label{e:basic-chordal-decomposition}
P(x) = \sum_{k=1}^{t} E_{\mathcal{C}_k}^\tr S_k(x) E_{\mathcal{C}_k}.
\end{equation}
Are these matrices always polynomial in $x$? Our first result gives a negative answer to this question for all matrix sizes $m \geq 3$, irrespective of the number $n$ of independent variables and of the sparsity graph of $P$.

\begin{proposition}\label{th:failure-basic-decomposition}
	Let $\mathcal{G}$ be a connected and not complete chordal graph with $m\geq 3$ vertices and maximal cliques $\mathcal{C}_1,\ldots,\mathcal{C}_t$. Fix any positive integer $n$. There exists an $n$-variate $m \times m$ polynomial matrix $P(x)$ with sparsity graph $\mathcal{G}$ that is strictly positive definite for all $x \in \mathbb{R}^n$, but cannot be written in the form~\cref{e:basic-chordal-decomposition} with positive semidefinite polynomial matrices $S_k(x)$.
\end{proposition}
\noindent 
The proof of this proposition, given in \cref{ss:counterexample}, relies on the following example. 

\begin{example}\label{ex:nondecomposable-3d-matrix}
    The $3\times 3$ univariate polynomial matrix
    \begin{equation} \label{eq:example1}
        P(x) = \begin{bmatrix}
             k+1+x^2& x+x^2 & 0 \\
             x+x^2 & k+2x^2 & x-x^2 \\
             0 &  x-x^2    & k+1+x^2
            \end{bmatrix}
            = 
            \begin{bmatrix}
            x & 1\\
            x  & x\\
            1 &  -x
            \end{bmatrix}
            \begin{array}{@{}c@{}}
            \begin{bmatrix}
            x  & x & 1\\
            1 & x & -x
            \end{bmatrix}
            \\ \mathstrut
            \end{array}
            + k I_3
    \end{equation}
    is globally positive semidefinite and SOS for all $k\geq 0$, and it is strictly positive definite if $k >0$.    
    Let us try to search for a basic decomposition of the form~\cref{e:basic-chordal-decomposition}. We need to find two $2 \times 2$ positive semidefinite polynomial matrices $S_1$ and $S_2$ such that $P(x) = E_{\mathcal{C}_1}^\tr S_1(x) E_{\mathcal{C}_1} + E_{\mathcal{C}_2}^\tr S_2(x) E_{\mathcal{C}_2}$. Equivalently, we need to find polynomials $a$, $b$, $c$, $d$, $e$ and $f$ such that
    \begin{equation}\label{e:ex-decomposition-basic}
        P(x) = \begin{bmatrix}
             a(x)& b(x) & 0 \\
             b(x) & c(x) & 0 \\
             0 &  0   & 0
            \end{bmatrix}
            +
            \begin{bmatrix}
             0& 0 & 0 \\
             0 & d(x) & e(x) \\
             0 &  e(x)    & f(x)
            \end{bmatrix},
    \end{equation}
    and such that the two matrices on the right-hand sides are positive semidefinite.
    Fixing
    $a(x)=k+1+x^2$, $b(x) = x+x^2$, $e(x)=x-x^2$, $f(x) = k+1+x^2$ and $d(x) = k + 2x^2 - c(x)$ to ensure the equality, positive semidefiniteness requires the traces and determinants of the $2\times 2$ nonzero blocks to be nonnegative, i.e.
    \begin{subequations}
    \label{eq:ex1_trace}
    \begin{gather}
        c(x) \geq 0,  \label{eq:ex1_trace1}\\
        k + 2x^2 - c(x) \geq 0, \label{eq:ex1_trace2}\\
     (k+1+x^2)c(x) - (x^4 + 2x^3+x^2)\geq 0,  \label{eq:ex1_trace3}\\
     x^4 + 2x^3 + (3k+1)x^2 + k^2 + k - (k+1+x^2)c(x) \geq 0. \label{eq:ex1_trace4}
     \end{gather}
    \end{subequations}
    If $c(x)$ is to be nonnegative, then it must be quadratic; otherwise,~\cref{eq:ex1_trace2} cannot hold for all $x$. In particular, we must have $c(x) = \alpha + 2x + x^2$ for some scalar $\alpha$ to ensure that the coefficients of $x^4$ and $x^3$ in~\cref{eq:ex1_trace3} and~\cref{eq:ex1_trace4} vanish, otherwise at least one of these conditions cannot hold for all $x \in \mathbb{R}$.  
    Then,~\cref{eq:ex1_trace1} and~\cref{eq:ex1_trace2} become
    $x^2 + 2 x + \alpha \geq 0$ and $x^2 - 2 x - \alpha + k \geq 0$, and hold if and only if $1\leq \alpha \leq k-1$.
    A suitable $\alpha$ therefore exists when $k\geq 2$, while the decomposition~\cref{e:ex-decomposition-basic} fails to exist if $0\leq k <2$ even though $P(x)$ is PSD for all such values of $k$ (and, in fact, positive definite if $k\neq 0$). \markendexample
\end{example}

Clique-based decompositions similar to~\cref{e:basic-chordal-decomposition} with polynomial matrices $S_k(x)$, however, do exist after multiplying $P(x)$ by a suitable SOS polynomial $\sigma(x)$. The next result generalizes the Hilbert--Artin theorem on the representation of nonnegative polynomial as sums of squares of rational functions~\cite{artin1927zerlegung}. Importantly, it establishes that each $S_k(x)$ is not just positive semidefinite, but SOS.

\begin{theorem} \label{th:weighted-chordal-decomposition}
	Let $P(x)$ be an $m \times m$ positive semidefinite polynomial matrix whose sparsity graph is chordal and has maximal cliques $\mathcal{C}_1,\ldots,\mathcal{C}_t$. There exist an SOS polynomial $\sigma(x)$ and SOS matrices 
	$S_k(x)$ of size $|\mathcal{C}_k| \times |\mathcal{C}_k|$ such that
	\begin{equation} \label{E:DecompositionSparseCone}
	\sigma(x) P(x)= \sum_{k=1}^{t} E_{\mathcal{C}_k}^\tr S_k(x) E_{\mathcal{C}_k}.
	\end{equation}
\end{theorem}
\noindent
The proof, given in \cref{ss:proof-weighted-chordal-decomposition}, extends a constructive proof of \cref{th:ChordalDecompositionTheorem} for standard PSD matrices with chordal sparsity~\cite{kakimura2010direct} using Schm\"udgen's diagonalization procedure for polynomial matrices~\cite{Schmudgen2009noncommutative} and the Hilbert--Artin theorem~\cite{artin1927zerlegung}. 

\begin{example}
    Consider once again the polynomial matrix $P(x)$ from \cref{ex:nondecomposable-3d-matrix}. Inequalities~{(\ref{eq:ex1_trace}a--d)} hold for the rational function $c(x) = (1+x)^2x^2(k+1+x^2)^{-1}$. We can therefore decompose
    \begin{equation}\label{e:rational-decomposition-P-example}
    P(x) = 
    (1+k+x^2)^{-1}
    \left[ E_{\mathcal{C}_1}^\tr S_1(x) E_{\mathcal{C}_1} + E_{\mathcal{C}_2}^\tr S_2(x) E_{\mathcal{C}_2} \right]
    \end{equation}
    where,  by construction, the polynomial matrices 
    \begin{align*}
        S_1(x) &:= \begin{bmatrix}
            (k+1+x^2)^2 & (k+1+x^2)(x+x^2)\\
            (k+1+x^2)(x+x^2) & (1+x)^2 x^2
            \end{bmatrix} 
        \\
        S_2(x) &:= \begin{bmatrix}
            k^2 + k + 3k x^2 + (1-x)^2x^2 & (k+1+x^2)(x-x^2) \\
            (k+1+x^2)(x-x^2)    & (k+1+x^2)^2
            \end{bmatrix} 
    \end{align*}
    are PSD for all $k \geq 0$. They are also SOS because the two concepts are equivalent for univariate polynomial matrices~\cite{aylward2007explicit}. 
    Rearranging~\cref{e:rational-decomposition-P-example} yields the decomposition of $P$ guaranteed by \cref{th:weighted-chordal-decomposition} with $\sigma(x) = k+1 + x^2$. \markendexample 
\end{example}

If $P(x)$ and its highest-degree homogeneous part are strictly positive definite on $\mathbb{R}^n$ and $\mathbb{R}^n\setminus\{0\}$, respectively, one can fix either $\sigma(x)=\|x\|^{2\nu}$ or $\sigma(x)=(1 +\|x\|^2)^{\nu}$ for a sufficiently large integer $\nu \geq 0$, where $\|x\|^2 := x_1^2 + \cdots + x_n^2$. Precisely, we have the following versions of Reznick's Positivstellensatz~\cite{Reznick1995} for sparse polynomial matrices, which follow from more general SOS chordal decomposition results on semialgebraic sets stated in the next section (cf. \cref{th:sparse-putinar-vasilescu-homog,th:sparse-putinar-vasilescu}).

\begin{theorem}\label{th:sparse-reznick-homog}
	Let $P(x)$ be an $m \times m$ homogeneous polynomial matrix whose sparsity graph is chordal and has maximal cliques $\mathcal{C}_1,\ldots,\mathcal{C}_t$. If $P$ is strictly positive definite on $\mathbb{R}^n \setminus\{0\}$, there exist an integer $\nu \geq 0$ and homogeneous SOS matrices $S_{k}(x)$ of size $|\mathcal{C}_k| \times |\mathcal{C}_k|$ such that
	\begin{equation}\label{e:sparse-reznick-homog}
	\|x\|^{2\nu} P(x) = \sum_{k=1}^t E_{\mathcal{C}_k}^\tr S_{k}(x) E_{\mathcal{C}_k}.
	\end{equation}
\end{theorem}

\begin{corollary}\label{th:sparse-reznick}
	Let $P(x) = \sum_{\abs{\alpha}\leq 2d} P_\alpha x^\alpha$ be an inhomogeneous $m \times m$ polynomial matrix of even degree $2d$ whose sparsity graph is chordal and has maximal cliques $\mathcal{C}_1,\ldots,\mathcal{C}_t$. If $P$ is strictly positive definite on $\mathbb{R}^n$ and its highest-degree homogeneous part $\sum_{\abs{\alpha}=2d}P_\alpha x^\alpha$ is strictly positive definite on $\mathbb{R}^n\setminus \{0\}$, there exist an integer $\nu \geq 0$ and SOS matrices $S_{k}(x)$ of size $|\mathcal{C}_k| \times |\mathcal{C}_k|$ such that
	\begin{equation}\label{e:sparse-reznick}
	(1+\|x\|^2)^\nu P(x) = \sum_{k=1}^t E_{\mathcal{C}_k}^\tr S_{k}(x) E_{\mathcal{C}_k}.
	\end{equation}
\end{corollary}

\begin{example} \label{ex:nondecomposable-multivariate}
    Let $q(x) = x_1^2 x_2^4 + x_1^4 x_2^2 - 3 x_1^2 x_2^2 + 1$ be the Motzkin polynomial~\cite{Motzkin1967}, which is nonnegative but not SOS~\cite[Example~3.7]{Laurent2009}. The polynomial matrix
    \begin{equation} \label{eq:counterex2}
        P(x) = \begin{bmatrix} 0.01(1+x_1^6+x_2^6)+q(x) & -0.01x_1 & 0 \\
                            -0.01x_1 & x_1^6+x_2^6+1 & -x_2 \\
                            0& -x_2 & x_1^6 + x_2^6 + 1 \end{bmatrix}
    \end{equation}
    is strictly positive definite on $\mathbb{R}^2$ (see \cref{app:motzkin-pd}), but is not SOS since $\varepsilon (1+x_1^6+x_2^6)+q(x)$ is not SOS unless $\varepsilon\gtrsim 0.01006$~\cite[Example~6.25]{Laurent2009}.  Nevertheless, since the highest-degree homogeneous part of $P$ is also positive definite on $\mathbb{R}^2\setminus\{0\}$, \cref{th:sparse-reznick} guarantees that $P$ can be decomposed as in~\cref{e:sparse-reznick} for a large enough exponent $\nu$. Here $\nu=1$ suffices, and $(1+ \|x\|^2) P(x) =  E_{\mathcal{C}_1}^\tr S_1(x) E_{\mathcal{C}_1} + E_{\mathcal{C}_2}^\tr S_2(x) E_{\mathcal{C}_2}$ with
    \begin{subequations}
    	\begin{equation}\label{e:motzkin-S1}
    	S_1(x) = \begin{bmatrix}
    	(1+\|x\|^2)q(x) & 0 \\ 0 & 0
    	\end{bmatrix}
    	+
    	\frac{1+\|x\|^2}{100}\begin{bmatrix}
    	1+x_1^6+x_2^6 & -x_1 \\ -x_1 & 100 x_1^2
    	\end{bmatrix}
    	\end{equation}
    	and
    	\begin{equation}\label{e:motzkin-S2}
    	S_2(x) = (1+\|x\|^2)\begin{bmatrix}
    	1 - x_1^2 + x_1^6 + x_2^6   &   -x_2 \\ -x_2 & 1+x_1^6+x_2^6
    	\end{bmatrix}.
    	\end{equation}
    \end{subequations}
    To see that these two matrices are SOS, observe that the first addend on the right-hand side of~\cref{e:motzkin-S1} is SOS because $(1+\|x\|^2)q(x) = (1-x_1^2x_2^2)^2 + x_2^2(1-x_1^2)^2 + x_1^2(1-x_2^2)^2 + \tfrac14(x_1^3x_2 - x_1x_2^3)^2 + \tfrac34 (x_1^3x_2 + x_1x_2^3 - 2x_1x_2)^2$, the second addend on the right-hand side of~\cref{e:motzkin-S1} is SOS because
    \begin{equation*}
        \begin{bmatrix}1+x_1^6+x_2^6 & -x_1 \\ -x_1 & 100 x_1^2\end{bmatrix}
        = H(x)H(x)^\tr
        \quad\text{with}\quad
        H(x) = \begin{bmatrix}1 & x_1^3 & 0 & x_2^3 \\ -x_1 & 0 & \sqrt{99}x_1 & 0\end{bmatrix},
    \end{equation*}
    and the matrix on the right-hand side of~\cref{e:motzkin-S2} is the sum of two univariate PSD (hence, SOS) matrices: setting $k=2/(3\sqrt{3})$, we have
     \begin{equation*}
	     \hspace{15pt}
         \begin{bmatrix}
         1 - x_1^2 + x_1^6 + x_2^6   &   -x_2 \\ -x_2 & 1+x_1^6+x_2^6
         \end{bmatrix}
         \!\!=\!\!
         \begin{bmatrix}
         k - x_1^2 + x_1^6   &   0 \\ 0 & x_1^6
         \end{bmatrix}
         \!+\!
         \begin{bmatrix}
         1 - k + x_2^6   &   -x_2 \\ -x_2 & 1+x_2^6
         \end{bmatrix}.
         \hspace{5pt}
         \text{\xqed{{\small$\blacksquare$}}}
     \end{equation*}
     
\end{example}

\subsection{Polynomial matrix decomposition on semialgebraic sets}
\label{ss:results-semialgebraic}

We now turn our attention to SOS chordal decompositions on basic semialgebraic sets $\mathcal{K}$ defined as in~\cref{e:semialgebgraic-set}. We say that $\mathcal{K}$ satisfies the Archimedean condition if there exist SOS polynomials $\sigma_0(x),\,\ldots,\,\sigma_q(x)$ and a scalar $r$ such that
\begin{equation}\label{e:archimedean-condition}
 \sigma_0(x) + g_1(x) \sigma_1(x) + \cdots + g_q(x) \sigma_q(x) = r^2 - \|x\|^2.
\end{equation}
This condition implies that $\mathcal{K}$  is compact because $r^2 - \|x\|^2$ is positive on $\mathcal{K}$. The converse is not always true~\cite{lasserre2010moments}, but can be ensured by adding the redundant inequality $r^2 - \|x\|^2 \geq 0$ to the definition~\cref{e:semialgebgraic-set} of $\mathcal{K}$ for a sufficiently large $r$.

\Cref{th:sparse-putinar} below guarantees that if a polynomial matrix is strictly positive definite on a compact $\mathcal{K}$ satisfying the Archimedean condition, then it admits a chordal decomposition in terms of weighted sums of SOS matrices supported on the cliques of the sparsity graph, where the weights are exactly the polynomials $g_1,\,\ldots,\,g_q$ used in the semialgebraic definition~\cref{e:semialgebgraic-set} of $\mathcal{K}$. This result extends Putinar's Positivstellensatz~\cite{putinar1993positive} to sparse polynomial matrices, and can be considered a sparsity-exploiting version of a Positivstellensatz for general (dense) polynomial matrices (see~\cite[Theorem 2.19]{lasserre2015introduction_book} and~\cite[Theorem 2]{scherer2006matrix}). 


\begin{theorem}\label{th:sparse-putinar}
	Let $\mathcal{K}$ be a compact semialgebraic set defined as in~\cref{e:semialgebgraic-set} that satisfies the Archimedean condition~\cref{e:archimedean-condition}, and let $P(x)$ be a polynomial matrix whose sparsity graph is chordal and has maximal cliques $\mathcal{C}_1,\ldots,\mathcal{C}_t$. If $P$ is strictly positive definite on $\mathcal{K}$, there exist SOS matrices $S_{j,k}(x)$ of size $|\mathcal{C}_k| \times |\mathcal{C}_k|$ such that
	\begin{equation}\label{e:sparse-matrix-sos}
	P(x) = \sum_{k=1}^t E_{\mathcal{C}_k}^\tr \bigg( S_{0,k}(x) + \sum_{j=1}^q g_j(x)S_{j,k}(x) \bigg) E_{\mathcal{C}_k}.
	\end{equation}
\end{theorem}
\noindent
The proof, given in \cref{ss:proof-sparse-putinar}, exploits the Cholesky algorithm for matrices with chordal sparsity, the Weierstrass polynomial approximation theorem, and the aforementioned Positivstellensatz for general polynomial matrices~\cite[Theorem~2]{scherer2006matrix}. 

\begin{example}\label{ex:sparse-putinar}
        The bivariate polynomial matrix
        \begin{equation} \label{eq:example_bivarite}
        P(x) :=
            \begin{bmatrix}
            1+2x_1^2-x_1^4 & x_1+x_1x_2-x_1^3 & 0\\
            x_1+x_1x_2-x_1^3 & 3+4x_1^2-3x_2^2 & 2x_1^2x_2-x_1x_2-2x_2^3\\
            0 & 2x_1^2x_2-x_1x_2-2x_2^3 & 1+x_2^2+x_1^2x_2^2-x_2^4
            \end{bmatrix}
    \end{equation}
    is not positive semidefinite globally (the first diagonal element is negative if $x_1$ is sufficiently large) but is strictly positive definite on the compact semialgebraic set 
    $\mathcal{K}=\{x \in\mathbb{R}^2:\;
    g_1(x) := 1 - x_1^2 \geq 0,\,
    g_2(x) := x_1^2 - x_2^2 \geq 0\}$. 
    This can be verified numerically by approximating the region of $\mathbb{R}^2$ where $P$ is positive definite (see \cref{f:bowtie-plot}), and an analytical certificate will be given below.
    \begin{figure}
        \centering
        \includegraphics[scale=0.65]{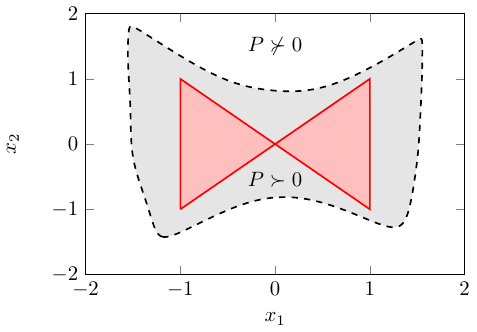}
        \caption{The semialgebraic set $\mathcal{K}$ considered in \cref{ex:sparse-putinar} (red shading, solid boundary), compared to the region of $\mathbb{R}^2$ where the matrix $P(x)$ in~\cref{eq:example_bivarite} is positive definite (grey shading, dashed boundary). On the boundary of this region, $P(x)$ is PSD but not definite.}
        \label{f:bowtie-plot}
    \end{figure}
    
    The semialgebraic set $\mathcal{K}$ satisfies the Archimedean condition~\cref{e:archimedean-condition} with $\sigma_0(x) = 0$, $\sigma_1(x)=2$, $\sigma_2(x)=1$ and $r=\sqrt{2}$. Therefore, \cref{th:sparse-putinar} guarantees that 
    \begin{equation}\label{e:example-putinar-decomposition}
    P(x) = 
    \sum_{k=1}^2
    E_{\mathcal{C}_k}^\tr \left[S_{0,k}(x) + g_1(x) S_{1,k}(x) + g_2(x) S_{2,k}(x) \right] E_{\mathcal{C}_k}
    \end{equation}
    for some $2 \times 2$ SOS matrices $S_{0,1}$, $S_{1,1}$, $S_{2,1}$, $S_{0,2}$, $S_{1,2}$ and $S_{1,2}$. Possible choices for these matrices are $S_{2,1}=0$, $S_{1,2}=0$ and
    \begin{align*}
        S_{0,1}(x) 
            &= I_2
             +\begin{bmatrix} x_1 \\ x_2 \end{bmatrix} 
              \begin{array}{@{}c@{}}\begin{bmatrix} x_1 & x_2\end{bmatrix} \\ \mathstrut\end{array}
        &
        S_{1,1}(x) 
        &= \begin{bmatrix} x_1 \\ 1 \end{bmatrix}
        \begin{array}{@{}c@{}}\begin{bmatrix} x_1 & 1 \end{bmatrix} \\ \mathstrut\end{array}
        \\
        S_{0,2}(x) 
            &= I_2
             +\begin{bmatrix} x_1 \\ -x_2 \end{bmatrix} 
              \begin{array}{@{}c@{}}\begin{bmatrix} x_1 & -x_2\end{bmatrix} \\ \mathstrut\end{array}
        &
        S_{2,2}(x) 
            &=\begin{bmatrix} 2 \\ x_2 \end{bmatrix}
             \begin{array}{@{}c@{}}\begin{bmatrix} 2 & x_2 \end{bmatrix}  \\ \mathstrut\end{array}.
    \end{align*}
    Since $S_{0,1}$ and $S_{0,2}$ are positive definite and all other addends in~\cref{e:example-putinar-decomposition} are PSD on $\mathcal{K}$,  we conclude in particular that $P(x)$ is positive definite on that set, as claimed initially.
    \markendexample
\end{example}

If $\mathcal{K}$ is not compact or does not satisfy the Archimedean condition, \cref{th:sparse-putinar} can be used to prove a similar decomposition result that applies to $(1+\|x\|^2)^\nu P$ with large enough exponent $\nu$, as long as $P$ has even degree and the behaviour of its leading term can be controlled. We start with the case in which $P$ is homogeneous and $\mathcal{K}$ is defined using homogeneous polynomial inequalities of even degree.

\begin{theorem}\label{th:sparse-putinar-vasilescu-homog}
	Let $\mathcal{K}$ be a semialgebraic set defined as in~\cref{e:semialgebgraic-set} with homogeneous polynomials $g_1,\ldots,g_q$ of even degree, and such that $\mathcal{K} \setminus\{0\}$ is nonempty. Let $P(x)$ be a homogeneous polynomial matrix of even degree whose sparsity graph is chordal and has maximal cliques $\mathcal{C}_1,\ldots,\mathcal{C}_t$. If $P$ is strictly positive definite on $\mathcal{K} \setminus\{0\}$, there exist an integer $\nu \geq 0$ and homogeneous SOS matrices $S_{j,k}(x)$ of size $|\mathcal{C}_k| \times |\mathcal{C}_k|$ such that
	\begin{equation}\label{e:sparse-putinar-vasilescu-homog}
	\|x\|^{2\nu} P(x) = \sum_{k=1}^t E_{\mathcal{C}_k}^\tr \bigg( S_{0,k}(x) + \sum_{j=1}^q g_j(x)S_{j,k}(x) \bigg) E_{\mathcal{C}_k}.
	\end{equation}
\end{theorem}
\noindent
This result, proven in \cref{ss:proof-sparse-putinar-vasilescu}, recovers Theorem 4 in~\cite{Dinh2021} when $P$ is dense.
If $P$ is not homogeneous, we find the following version of the Putinar--Vasilescu Positivstellensatz~\cite{PutinarVasilescu1999} for sparse polynomial matrices, which is a sparsity-exploiting formulation of a recent result for general (dense) matrices~\cite[Corollary 3]{Dinh2021}.

\begin{corollary}\label{th:sparse-putinar-vasilescu}
	Let $\mathcal{K}$ be a semialgebraic set defined as in~\cref{e:semialgebgraic-set}, and let $P(x)= \sum_{|\alpha|\leq 2d_0} P_{\alpha} x^{\alpha}$ be an inhomogeneous polynomial matrix of even degree $2d_0$ whose sparsity graph is chordal and has maximal cliques $\mathcal{C}_1,\ldots,\mathcal{C}_t$. If $P$ is strictly positive definite on $\mathcal{K}$ and 
	its highest-degree homogeneous part $\sum_{|\alpha| = 2d_0} P_{\alpha} x^{\alpha}$ is strictly positive definite on $\mathbb{R}^n \setminus \{0\}$, there exist an integer $\nu \geq 0$ and SOS matrices $S_{j,k}(x)$ of size $|\mathcal{C}_k| \times |\mathcal{C}_k|$ such that
	\begin{equation}\label{e:sparse-putinar-vasilescu}
	(1+\|x\|^2)^\nu P(x) = \sum_{k=1}^t E_{\mathcal{C}_k}^\tr \bigg( S_{0,k}(x) + \sum_{j=1}^q g_j(x)S_{j,k}(x) \bigg) E_{\mathcal{C}_k}.
	\end{equation}
\end{corollary}

\begin{proof}
Set $Q(x,y) = y^{2d_0}P(x/y)$, $d_j = \lceil\frac12 \deg(g_j) \rceil$, $h_j(x,y) = y^{2d_j} g_j(x/y)$ for all $j=1,\ldots,q$, and $\mathcal{K}' = \{(x,y): \,h_j(x,y) \geq 0,\, j=1,\ldots,q\}$. The polynomial matrix $Q$ and the polynomials $h_j$ are homogeneous of even degree, and satisfy $Q(x,1) = P(x)$ and $h_j(x,1) = g_j(x)$ for all $j = 1, \ldots, q$. 
%
Furthermore, $Q$ is positive definite on $\mathcal{K}'\setminus \{(0,0)\}$ because $Q(x,0)=\sum_{|\alpha| = 2d_0} P_{\alpha} x^{\alpha}$ is positive definite by assumption, while if $(x,y) \in \mathcal{K}'$ with $y\neq0$, then $x/y \in \mathcal{K}$ and $Q(x,y)$ is positive definite because so is $P(x/y)$.
%
Applying \cref{th:sparse-putinar-vasilescu-homog} to $Q$ and $\mathcal{K}'$, noting that $\|(x,y)\|^2 = y^2 + \|x\|^2$, setting $y=1$, and recalling that $Q(x,1) = P(x)$ yields~\cref{e:sparse-putinar-vasilescu}.
\end{proof}

\begin{remark}
	Setting $g_1 = \cdots = g_q \equiv 0$ in \cref{th:sparse-putinar-vasilescu-homog,th:sparse-putinar-vasilescu} immediately yields \cref{th:sparse-reznick-homog,th:sparse-reznick} for the global case $\mathcal{K}=\mathbb{R}^n$ (observe that a globally PSD homogeneous polynomial matrix must have even degree).
\end{remark}

%% file: ss-convex-pmi.tex
The decomposition results in \cref{ss:results-semialgebraic,ss:results-global} can be used to construct hierarchies of sparsity-exploiting SOS reformulations for the optimization problem~\cref{e:infinite-sdp-intro} that produce feasible vectors $\lambda$ and upper bounds on its optimal value $B^*$.

Specifically, fix any two integers $\nu$ and $d$ satisfying $\nu \geq 0$ and 
$2d \geq \max\{\deg(P), \deg(g_1),$ $ \ldots, \deg(g_q)\} + 2\nu,$ 
and consider the SOS optimization problem
\begin{align}
B_{d,\nu}^* := \inf_{\lambda,\, S_{j,k}} \;&
b(\lambda)
\nonumber \\[-1ex]
\text{s.t.}\; 
&\sigma(x)^{\nu} P(x,\lambda) = \sum_{k=1}^t E_{\mathcal{C}_k}^\tr \bigg( S_{0,k}(x) + \sum_{j=1}^m g_j(x)S_{j,k}(x) \bigg) E_{\mathcal{C}_k}, 
\nonumber \\ \label{e:sos-reformulation-general}
& S_{j,k} \in \Sigma_{2d_j}^{\abs{\mathcal{C}_k}} \quad \forall j = 0, \ldots, q, \; \forall k = 1, \ldots, t,
\end{align}
where $\Sigma_{2\omega}^m$ denotes the cone of $n$-variate $m \times m$ SOS matrices of degree $2\omega$, $d_0 := d$, $d_j := d - \lceil\frac12 \deg(g_j) \rceil$ for each $j = 1, \ldots, q$, and either $\sigma(x) = \|x\|^2$ or $\sigma(x) = 1+ \|x\|^2$ depending on whether $P$ is homogeneous in $x$ or not. For each choice of $\nu$ and $d$, problem~\cref{e:sos-reformulation-general} can be recast as an SDP~\cite{gatermann2004symmetry,kojima2003sums,parrilo2013semidefinite,scherer2006matrix} and solved using a wide range of algorithms. The optimal $\lambda$ is clearly feasible for~\cref{e:infinite-sdp-intro}, so $B_{d,\nu}^* \geq B^*$. 

The nontrivial and far-reaching implication of the decomposition theorems presented in \cref{ss:results-semialgebraic,ss:results-global}  is that the SOS problem~\cref{e:sos-reformulation-general} is asymptotically exact as $d$ or $\nu$ are increased, provided that the original problem~\cref{e:infinite-sdp-intro} satisfies suitable technical conditions and is strictly feasible.
For instance, the sparsity-exploiting version of Putinar's Positivstellensatz in  \cref{th:sparse-putinar} leads to the following result. 
\begin{theorem}\label{th:sos-program-convergence_putinar}
	Let $\mathcal{K}$ be a compact basic semialgebraic set defined as in~\cref{e:semialgebgraic-set} that satisfies the Archimedean condition~\cref{e:archimedean-condition}, and let $B^*$ and $B_{d,\nu}^*$ be as in \cref{e:infinite-sdp-intro,e:sos-reformulation-general}. If there exists $\lambda_0 \in \mathbb{R}^\ell$ such that $P(x; \lambda_0)$ is strictly positive definite on $\mathcal{K}$, then $B_{d,0}^* \to B^*$ from above as $d \to \infty$. 
\end{theorem}
\begin{proof}
	It suffices to show that, for any $\varepsilon>0$, there exists $d$ such that $B^* \leq B_{d,0}^* \leq B^* + 2\varepsilon$.
	If $\lambda_0$ is optimal for~\cref{e:infinite-sdp-intro}, \cref{th:sparse-putinar} guarantees that $\lambda_0$ is feasible for~\cref{e:sos-reformulation-general} for $\nu=0$ (observe that $[\sigma(x)]^0 \equiv 1$) and some sufficiently large $d$. Since $b(\lambda_0) = B^* \leq B_{d,0}^* \leq b(\lambda_0)$, we obtain $B_{d,0}^*=B^*$. In particular, if the minimizer of~\cref{e:infinite-sdp-intro} is strictly feasible, then the convergence $B_{d,0}^* \to B^*$ is finite.
	
	If $\lambda_0$ is not optimal, fix $\varepsilon>0$ and let $\lambda_\varepsilon$ be an $\varepsilon$-suboptimal feasible point for~\cref{e:infinite-sdp-intro} such that 
	$b(\lambda_\varepsilon) \leq B^* + \varepsilon < b(\lambda_0)$.
	Fix $\lambda = (1-\gamma) \lambda_\varepsilon + \gamma \lambda_0$ for some $\gamma \in (0,1)$ to be determined. Since $P(x,\lambda_0)$ is strictly positive definite on $\mathcal{K}$ and $P(x,\lambda_\varepsilon)$ is PSD on the same set, the matrix
	$P(x,\lambda) = (1-\gamma)  P(x,\lambda_\varepsilon) + \gamma  P(x,\lambda_0)$
	is strictly positive definite on $\mathcal{K}$ and \cref{th:sparse-putinar} guarantees that $\lambda$ is feasible for~\cref{e:sos-reformulation-general} when $d$ is sufficiently large. Given such $d$, we can use the inequality $B^*\leq B_{d,0}^*$ and the convexity of the cost function $b$ to estimate
	\begin{multline*}
	B^* \leq B_{d,0}^*
	\leq b(\lambda)
	= b\left( (1-\gamma)   \lambda_\varepsilon + \gamma  \lambda_0 \right)
	\leq (1-\gamma)  b(\lambda_\varepsilon) + \gamma  b(\lambda_0)\\
	\leq (1-\gamma) B^* + (1-\gamma) \varepsilon + \gamma  b(\lambda_0)
	= B^* + \varepsilon + \gamma   \big[ b(\lambda_0) - B^* - \varepsilon \big].
	\end{multline*}
	The term in square brackets is strictly positive by construction, so we can fix	$\gamma = \varepsilon/[b(\lambda_0) - B^* - \varepsilon]$
	and conclude that $B^* \leq B_\nu^* \leq B^* + 2\varepsilon$, as required.
\end{proof}

If $\mathcal{K}$ is not compact or does not satisfy the Archimedean condition, similar arguments that use \cref{th:sparse-putinar-vasilescu-homog,th:sparse-putinar-vasilescu} instead of \cref{th:sparse-putinar} (omitted for brevity) give asymptotic convergence results provided that $P$ satisfies additional conditions. For homogeneous problems of even degree, strict feasiblity suffices.

\begin{theorem}\label{th:sos-program-convergence-homog}
	Let $\mathcal{K}$ be a basic semialgebraic set defined as in~\cref{e:semialgebgraic-set}, and let $B^*$ and $B_{d,\nu}^*$ be as in \cref{e:infinite-sdp-intro,e:sos-reformulation-general}. Suppose that $P(x,\lambda)$ and the polynomials $g_1,\ldots,g_q$ defining $\mathcal{K}$ are homogeneous of even degree in $x$ for all $\lambda$. If there exists $\lambda_0 \in \mathbb{R}^\ell$ such that $P(x; \lambda_0)$ is strictly positive definite on $\mathcal{K}\setminus \{0\}$, then $B_{d,\nu}^* \to B^*$ from above  as $\nu \to \infty$ with $d = \nu + \frac12 \max\{\deg(P), \deg(g_1), \ldots, \deg(g_q)\}$ and $\sigma(x) = \|x\|^2$.
\end{theorem}

For inhomogeneous problems, instead, we require additional control on the leading homogeneous part of $P(x,\lambda)$ for all $\lambda$.

\begin{theorem}\label{th:sos-program-convergence-inhomog}
	Let $\mathcal{K}$ be a basic semialgebraic set defined as in~\cref{e:semialgebgraic-set}, and let $B^*$ and $B_{d,\nu}^*$ be as in \cref{e:infinite-sdp-intro,e:sos-reformulation-general}. Suppose that $P(x,\lambda) = \sum_{\abs{\alpha}\leq 2d} P_\alpha(\lambda) x^\alpha$ is an inhomogeneous polynomial matrix of even degree $2d$ such that $\sum_{\abs{\alpha}= 2d} P_\alpha(\lambda) x^\alpha$ is positive semidefinite on $\mathbb{R}^n$ for all $\lambda \in \mathbb{R}^\ell$. 
	If there exists $\lambda_0 \in \mathbb{R}^\ell$ such that $P(x; \lambda_0)$ is strictly positive definite on $\mathcal{K}$ and such that $\sum_{\abs{\alpha}= 2d} P_\alpha(\lambda_0) x^\alpha$ is strictly positive definite on $\mathbb{R}^n \setminus\{0\}$, 
	then $B_{d,\nu}^* \to B^*$ from above  as $\nu \to \infty$ with $d = \nu + \lceil\frac12 \max\{\deg(P), \deg(g_1), \ldots, \deg(g_q)\}\rceil$ and $\sigma(x) = 1+\|x\|^2$.
\end{theorem}

\begin{remark}
    \Cref{th:sos-program-convergence-homog,th:sos-program-convergence-inhomog} apply also when $\mathcal{K} \equiv \mathbb{R}^n$, in which case they can be deduced from~\cref{th:sparse-reznick-homog,th:sparse-reznick}. Thus, when $\mathcal{K} \equiv \mathbb{R}^n$ the SOS multipliers $S_{j,k}(x)$ for $j = 1, \ldots, q$ and $k = 1, \ldots, t$ in~\cref{e:sos-reformulation-general} can be set to zero.
\end{remark}

%% file: ss-csp-relation.tex
The SOS chordal decomposition theorems stated in \cref{s:main-results} can be used to derive new existence results for sparsity-exploiting SOS decompositions of certain families of correlatively sparse polynomials~\cite{waki2006sums,lasserre2006convergent,grimm2007note}.
A polynomial $$p(x,y) = \sum_{\alpha,\beta} c_{\alpha,\beta} \,x^\alpha y^\beta,$$ with independent variables $x=(x_1,\ldots,x_n)$ and $y=(y_1,\ldots,y_m)$ and coefficients $c_{\alpha,\beta} \in \mathbb{R}$, is \textit{correlatively sparse with respect to $y$} if the variables $y_1,\ldots,y_m$ are sparsely coupled, meaning that the $m \times m$ coupling matrix ${\rm CSP}_y(p)$ with entries
\begin{equation} \label{eq:csp}
[{\rm CSP}_y(p)]_{ij} = 
\begin{cases}
    1 &\text{if } i = j \text{ or } \exists \beta: \beta_i \beta_j \neq 0 \text{ and } c_{\alpha,\beta} \neq 0 \\
    0 &\text{otherwise}
\end{cases}
\end{equation}
is sparse. For example, the polynomial $p(x,y) = x_1^2x_2 y_1^2 +y_1y_2- x_2 y_2y_3 + y_4^4$ with $n=2$ and $m=4$ is correlatively sparse with respect to $y$ and
\def\arraystretch{0.8}
\begin{equation*}
    {\rm CSP}_y(x_1^2x_2 y_1^2 + y_1y_2 - x_2 y_2y_3 + y_4^4) = \begin{bmatrix}\\[-8pt]
    1 & 1 & 0 &0\\ 1 & 1 & 1 &0\\0 & 1 & 1 &0\\  0 & 0 & 0 & 1\\[-0.25pt]
    \end{bmatrix}.
\end{equation*}
The sparsity graph of the coupling matrix ${\rm CSP}_y(p)$ is known as the {correlative sparsity graph} of $p$, and we say that $p(x,y)$ has chordal correlative sparsity with respect to $y$ if its correlative sparsity graph is chordal.

To exploit correlative sparsity when attempting to verify the nonnegativity of $p(x,y)$, one looks for an SOS decomposition in the form~\cite{waki2006sums,lasserre2006convergent}
\begin{equation}\label{e:corr-sparse-decomposition}
    p(x,y) = \sum_{k=1}^t \sigma_k\!\left(x,y_{\mathcal{C}_k}\right),
\end{equation}
where $\mathcal{C}_1,\ldots,\mathcal{C}_t$ are the maximal cliques of the correlative sparsity graph and each $\sigma_k$ is an SOS polynomial that depends on $x$ and on the subset $y_{\mathcal{C}_k} = E_{\mathcal{C}_k} y$ of $y$ indexed by $\mathcal{C}_k$. For instance, with $m=3$ and two cliques $\mathcal{C}_1 = \{1,2\}$ and $\mathcal{C}_2 = \{2,3\}$ we have $y_{\mathcal{C}_1} = (y_1, y_2)$ and  $y_{\mathcal{C}_2} = (y_2, y_3)$. 


In general, the existence of the sparse SOS representation~\cref{e:corr-sparse-decomposition} is only sufficient to conclude that $p(x,y)$ is nonnegative: Example 3.8 in~\cite{Nie2008} gives a nonnegative (in fact, SOS) correlatively sparse polynomial that cannot be decomposed as in~\cref{e:corr-sparse-decomposition}. Nevertheless, our SOS chordal decomposition theorems from \cref{s:main-results} imply that sparsity-exploiting SOS decompositions do exist for polynomials $p(x,y)$ that are quadratic and correlatively sparse with respect to $y$. This is because any polynomial $p(x,y)$ that is correlatively sparse, quadratic, and (without loss of generality) homogeneous with respect to $y$ can be expressed as $p(x,y)=y^\tr P(x) y$ for some polynomial matrix $P(x)$ whose sparsity graph coincides with the correlative sparsity graph of $p(x,y)$. 
Using this observation, we can ``scalarize'' \cref{th:weighted-chordal-decomposition,,th:sparse-reznick-homog,,th:sparse-putinar,,th:sparse-putinar-vasilescu-homog} to obtain the following statements.

\begin{corollary} \label{corollary:globalcase}
    Let $p(x,y)=\sum_{\alpha, |\beta|\leq 2} c_{\alpha,\beta}x^\alpha y^\beta$ be nonnegative on $\mathbb{R}^n \times \mathbb{R}^m$, 
    quadratic and correlatively sparse in $y$, and such that $\sum_{\alpha, |\beta| = 2} c_{\alpha,\beta}x^\alpha y^\beta$ is nonnegative globally.
    If the correlative sparsity graph is chordal with maximal cliques $\mathcal{C}_1,\ldots,\mathcal{C}_t$, there exist an SOS polynomial $\sigma_0(x)$ and 
    SOS polynomials $\sigma_k(x,y_{\mathcal{C}_k})$ quadratic in the second argument
    such that
    $\sigma_0(x) p(x,y) = \sum_{k=1}^t \sigma_k\!\left(x, y_{\mathcal{C}_k}\right).$
\end{corollary}
\begin{proof}
    Assume first that $p$ is homogeneous in $y$ and write $p(x,y)=y^\tr P(x) y$, where $P(x)$ is positive semidefinite globally and has the same sparsity pattern as the correlative sparsity matrix ${\rm CSP}_y(p)$. \Cref{th:weighted-chordal-decomposition} guarantees that
    \begin{align*} 
    \sigma_0(x) p(x,y)  = y^\tr \left[ \sigma_0(x) P(x) \right] y
                   =  y^\tr \bigg(\sum_{k=1}^{t} E_{\mathcal{C}_k}^\tr S_k(x) E_{\mathcal{C}_k}\bigg) y
                   = \sum_{k=1}^t y_{\mathcal{C}_k}^\tr S_k(x) y_{\mathcal{C}_k}
    \end{align*}
    for some SOS polynomial $\sigma_0(x)$ and SOS polynomial matrices $S_k(x)$.
    Setting $\sigma_k(x,y_{\mathcal{C}_k}) :=  y_{\mathcal{C}_k}^\tr S_k(x) y_{\mathcal{C}_k}$
    gives the desired decomposition. When $p$ is not homogeneous, the result follows from a relatively straightforward homogenization argument described in \cref{app:homogenization}.
\end{proof}

\begin{corollary}\label{corollary:global_even}
    Let $p(x,y)=\sum_{|\alpha| = 2d, |\beta| \leq 2} c_{\alpha,\beta}x^\alpha y^\beta$ be homogeneous with degree $2d$ in $x$, and both quadratic and correlatively sparse in $y$. Suppose that
    \begin{enumerate}[1),topsep=2pt,noitemsep]
    	\item The correlative sparsity graph is chordal with maximal cliques $\mathcal{C}_1,\ldots,\mathcal{C}_t$;
    	\item $\sum_{|\alpha|=2d, |\beta| = 2} c_{\alpha,\beta}x^\alpha y^\beta > 0$ for all $(x,y)\neq (0,0)$;
    	\item If $p$ is not homogeneous in $y$, then $p(x,y)>0$ for all $x \neq 0$ and $y\in\mathbb{R}^m$.
    \end{enumerate}
    Then, there exist an integer $\nu\geq 0$ and SOS polynomials $\sigma_k(x,y_{\mathcal{C}_k})$ quadratic in the second argument such that
    %
    $\|x\|^{2\nu} p(x,y) = \sum_{k=1}^t \sigma_k\!\left(x,y_{\mathcal{C}_k}\right).$
    %
\end{corollary}
\begin{proof}
    If $p$ is homogeneous in $y$, write $p(x,y)=y^\tr P(x) y$,  observe that $P$ is strictly positive definite for all $x \in \mathbb{R}^n \setminus \{0\}$, apply \cref{th:sparse-reznick-homog} to $P$, and proceed as in the proof of \cref{corollary:globalcase}. If $p$ is not homogeneous, use a homogenization argument similar to that in \cref{app:homogenization}.
\end{proof}

\begin{corollary}\label{corollary:compact}
    Let $p(x,y)=\sum_{|\alpha| \leq d, |\beta| \leq 2} c_{\alpha,\beta}x^\alpha y^\beta$ be quadratic and correlatively sparse in $y$. Further, let $\mathcal{K}$ be a semialgebraic set defined as in~\cref{e:semialgebgraic-set}. Suppose that
    \begin{enumerate}[1),topsep=2pt,noitemsep]
    	\item The correlative sparsity graph is chordal with maximal cliques $\mathcal{C}_1,\ldots,\mathcal{C}_t$,
        \item $\sum_{|\alpha|\leq d, |\beta| = 2} c_{\alpha,\beta}x^\alpha y^\beta > 0$ for all $x\in \mathcal{K}$ and $y \in \mathbb{R}^m\setminus \{0\}$,
        \item If $p$ is not homogeneous in $y$, then $p(x,y)>0$ for all $x\in \mathcal{K}$ and $y \in \mathbb{R}^m$. 
    \end{enumerate}
    Then: 
    \begin{enumerate}[topsep=2pt,noitemsep,widest=99,leftmargin =*]
    	\item[i)] If $\mathcal{K}$ is compact and satisfies the Archimedean condition~\cref{e:archimedean-condition}, there exist SOS polynomials $\sigma_{j,k}(x,y_{\mathcal{C}_k})$, quadratic in the second argument, such that
    \begin{equation*}
    p(x,y) = \sum_{k=1}^t \bigg[ \sigma_{0,k}\!\left(x, y_{\mathcal{C}_k}\right) + \sum_{j=1}^q g_{j}(x)\sigma_{j,k}\!\left(x, y_{\mathcal{C}_k}\right) \bigg].
    \end{equation*}
    	\item[ii)] If $p$ and the polynomials $g_1,\ldots,g_q$ defining $\mathcal{K}$ are homogeneous of even degree in $x$, the set $\mathcal{K} \setminus \{0\}$ is nonempty, and conditions 2) and 3) above hold for $x\in \mathcal{K} \setminus \{0\}$,
    	there exist an integer $\nu \geq 0$ and SOS polynomials  $\sigma_{j,k}(x,y_{\mathcal{C}_k})$, quadratic in the second argument,
        such that
        \begin{equation*}
        \|x\|^{2\nu} p(x,y) = \sum_{k=1}^t \bigg[ \sigma_{0,k}\!\left(x, y_{\mathcal{C}_k}\right) + \sum_{j=1}^q g_{j}(x)\sigma_{j,k}\!\left(x, y_{\mathcal{C}_k}\right) \bigg].
        \end{equation*}
        \end{enumerate}
\end{corollary}
\begin{proof}
    If $p$ is homogeneous in $y$, write $p(x,y)=y^\tr P(x)y$ for a polynomial matrix $P(x)$ with chordal sparsity graph. The strict positivity of $p$ for all nonzero $y$ implies that $P$ is strictly positive definite on $\mathcal{K}$. Therefore, we can apply \cref{th:sparse-putinar} for statement i) and \cref{th:sparse-putinar-vasilescu-homog} for statement ii), and proceed as in the proof of \cref{corollary:globalcase} to conclude the proof.
    If $p$ is not homogeneous in $y$, one can use a homogenization argument similar to that in \cref{app:homogenization}.
\end{proof}

\Cref{corollary:compact} specializes, but appears not to be a particular case of, an SOS representation result for correlative sparse polynomials proved by Lasserre~\cite[Theorem~3.1]{lasserre2006convergent}. Similarly, \cref{corollary:globalcase,corollary:global_even} specialize recent results in~\cite{mai2020sparse}. In particular, although our statements apply only to polynomials $p(x,y)$ that are quadratic and correlatively sparse with respect to $y$ rather than to general ones, they provide explicit and tight degree bounds on the quadratic variables that cannot be deduced directly from the (more general) results in the references.
For example, let $\mathcal{K}$ be as in~\cref{e:semialgebgraic-set}, suppose that the Archimedean condition~\cref{e:archimedean-condition} holds, and suppose that $p(x,y)$ is quadratic, homogeneous, and correlatively sparse in $y$ with a chordal correlative sparsity graph. If $p$ is strictly positive for all $x \in \mathcal{K}$ and all $y \in \mathbb{R}^m\setminus\{0\}$, then in particular it is so on the basic semialgebraic set $\mathcal{K}':=\{(x,y) \in \mathcal{K} \times \mathbb{R}^m: \pm(1-y_1^2) \geq 0, \ldots, \pm(1-y_m^2) \geq 0\}$.
%
%
This set also satisfies the Archimedean condition, so one can use Theorem 3.1 in~\cite{lasserre2006convergent} to represent $p$ as
%
\begin{equation}
p(x,y) = \sum_{k=1}^t \bigg[ \sigma_{0k}\!\left(x, y_{\mathcal{C}_k}\right) 
+ \sum_{j=1}^q g_{j}(x)\sigma_{jk}\!\left(x, y_{\mathcal{C}_k}\right)
+ \sum_{\ell \in \mathcal{C}_k} \rho_{k\ell}(x,y_{\mathcal{C}_k}) (1 - y_\ell^2) \bigg]
\label{e:sparse-correlative-expansion-quadratic-example}
\end{equation}
for some SOS polynomials $\sigma_{jk}$ and some polynomials $\rho_{k\ell}$, not necessarily SOS. 
\Cref{corollary:compact} enables one to go further and conclude that one may take $\rho_{k\ell}\equiv 0$ and
$\sigma_{jk}\!\left(x, y_{\mathcal{C}_k}\right) =  y_{\mathcal{C}_k}^\tr S_{jk}(x) y_{\mathcal{C}_k}$ for some SOS matrices $S_{jk}$. These restrictions could probably be deduced starting from~\cref{e:sparse-correlative-expansion-quadratic-example}, but our approach based on the SOS chordal decomposition of sparse polynomial matrices makes them almost immediate.

%% file: ss-examples-numerical.tex

We now give numerical examples demonstrating the practical performance of the sparsity-exploiting SOS reformulations of the optimization problem~\cref{e:infinite-sdp-intro} introduced in \cref{ss:applications-to-PMI-optimiz}. All examples were implemented on a PC with a 2.2 GHz Intel Core i5 CPU and 12GB of RAM, using the SDP solver {MOSEK}~\cite{andersen2000mosek} and a customized version of the MATLAB optimization toolbox {YALMIP}~\cite{lofberg2004yalmip,lofberg2009pre}. The toolbox and all scripts used to generate the results presented below are available from \href{https://github.com/aeroimperial-optimization/aeroimperial-yalmip}{https://github.com/aeroimperial-optimization/aeroimperial-yalmip} and 
\href{https://github.com/aeroimperial-optimization/sos-chordal-decomposition-pmi}{https://github.com/aeroimperial-optimization/sos-chordal-decomposition-pmi}.

\subsection{Approximation of global polynomial matrix inequalities}
\label{example:quartic_polynomial_matrix}
Our first numerical experiment illustrates the computational advantage of our sparsity-exploiting SOS reformulation for a problem with a global polynomial matrix inequality.   
Fix an integer $\omega \geq 1$ and consider the $3\omega \times 3\omega$ tridiagonal polynomial matrix $P_\omega=P_\omega(x,\lambda)$, parameterized by $\lambda \in \mathbb{R}^2$, given by
\begin{equation*}
    P_\omega
    =
    \begin{bmatrix}
    \lambda_2 x_1^4+x_2^4 & \lambda_1 x_1^2 x_2^2\\
    \lambda_1 x_1^2 x_2^2 & \lambda_2 x_2^4+x_3^4 & \lambda_2 x_2^2 x_3^2\\
     & \lambda_2 x_2^2 x_3^2 & \lambda_2 x_3^4 + x_1^4 & \lambda_1 x_1^2 x_3^2\\
     &  & \lambda_1 x_1^2 x_3^2 & \lambda_2 x_1^4+x_2^4 & \lambda_2 x_1^2 x_2^2\\
     &  &  & \lambda_2 x_1^2 x_2^2 & \lambda_2 x_2^4+x_3^4 &  \ddots\\
     &  &  &  & \ddots &  \ddots & \lambda_i x_2^2 x_3^2\\
     &  &  &  &  & \lambda_i x_2^2 x_3^2 & \lambda_2 x_3^4 + x_1^4
    \end{bmatrix}\!,
\end{equation*}
where $i=1$ if $3\omega$ is even and $i = 2$ otherwise. Its sparsity graph is chordal with vertices $\mathcal{V}=\{1,\,\ldots,\,3\omega\}$, edges $\mathcal{E} = \{(1,2),\,(2,3),\,\ldots,\,(3\omega-1,3\omega)\}$, and maximal cliques $\mathcal{C}_1 =\{1, 2\}$, $\mathcal{C}_2 =\{2, 3\}$, $\ldots$ , $\mathcal{C}_{3\omega-1} =\{3\omega-1, 3\omega\}$. Observe that $P_\omega(x)$ is homogeneous for all $\lambda$, and it is positive definite on $\mathbb{R}^3\setminus\{0\}$ when $\lambda=(0,0)$.

First, we illustrate how \cref{th:sparse-reznick-homog} enables one to approximate the set of vectors $\lambda$ for which $P_\omega$ is PSD globally,
\begin{equation*} 
    \mathcal{F}_\omega=\{\lambda \in \mathbb{R}^2 :\; P_\omega(x,\lambda) \succeq 0 \quad \forall x \in \mathbb{R}^3\}.
\end{equation*}
%
%
Define two hierarchies of subsets of $\mathcal{F}_\omega$, indexed by a nonnegative integer $\nu$, as
\begin{subequations}
\begin{gather} 
    \label{eq:ex_feasible_region_sos}
    \mathcal{D}_{\omega,\nu} := \left\{ \lambda \in \mathbb{R}^2:\; \|x\|^{2\nu} P_\omega(x,\lambda) \text{ is SOS}\right\},
\\
    \label{eq:ex_feasible_region_sos_decomposition}
    \mathcal{S}_{\omega,\nu} := \bigg\{\lambda \in \mathbb{R}^2:\; \|x\|^{2\nu} P_\omega(x,\lambda) =  \sum_{k=1}^{3\omega - 1} E_{\mathcal{C}_k}^\tr S_k(x) E_{\mathcal{C}_k},
    S_k(x)\text{ is SOS}\bigg\}.
\end{gather}
\end{subequations}
The sets $\mathcal{D}_{\omega,\nu}$ are defined using the standard (dense) SOS constraint~\cref{e:matrix-sos}, while the sets $\mathcal{S}_{\omega,\nu}$ use the sparsity-exploiting nonnegativity certificate in~\cref{th:sparse-reznick-homog}.
For each $\nu$ we have ${\mathcal{S}}_{\omega,\nu} \subseteq {\mathcal{D}}_{\omega,\nu} \subseteq \mathcal{F}_\omega$, and the inclusions are generally strict. This is confirmed by the (approximations to) the first few sets $\mathcal{D}_{2,\nu}$ and $\mathcal{S}_{2,\nu}$ shown in \cref{fig:ex_feasible_region}, which were obtained by maximizing the linear cost function $\lambda_1\,\cos\theta + \lambda_2\,\sin\theta$ for 1000 equispaced values of $\theta$ in the interval $[0,\pi/2]$ and exploiting the $\lambda_1 \mapsto - \lambda_1$ symmetry of $\mathcal{D}_{2,\nu}$ and $\mathcal{S}_{2,\nu}$. (Computations for $\mathcal{S}_{2,1}$ were ill-conditioned, so the results are not reported.) 
On the other hand, for any choice of $\omega$, \cref{th:sparse-reznick-homog} guarantees that any $\lambda$ for which $P_\omega$ is positive definite belongs to $\mathcal{S}_{\omega,\nu}$ for sufficiently large $\nu$. Thus, the sets $\mathcal{S}_{\omega,\nu}$ can approximate $\mathcal{F}_\omega$ arbitrarily accurately 
in the sense that 
any compact subset of 
the interior of $\mathcal{F}_\omega$
is included in $\mathcal{S}_{\omega,\nu}$ for some sufficiently large integer $\nu$. The same is true for the sets $\mathcal{D}_{\omega,\nu}$ since $\mathcal{S}_{\omega,\nu} \subseteq \mathcal{D}_{\omega,\nu}$. Once again, this is confirmed by our numerical results for $\omega=2$ in~\cref{fig:ex_feasible_region}, which suggest that
$\mathcal{S}_{2,3} = \mathcal{D}_{2,2} = \mathcal{F}_2$. 

\begin{figure}
    \centering
    \subfigure[]{
    \includegraphics[scale = 0.45]{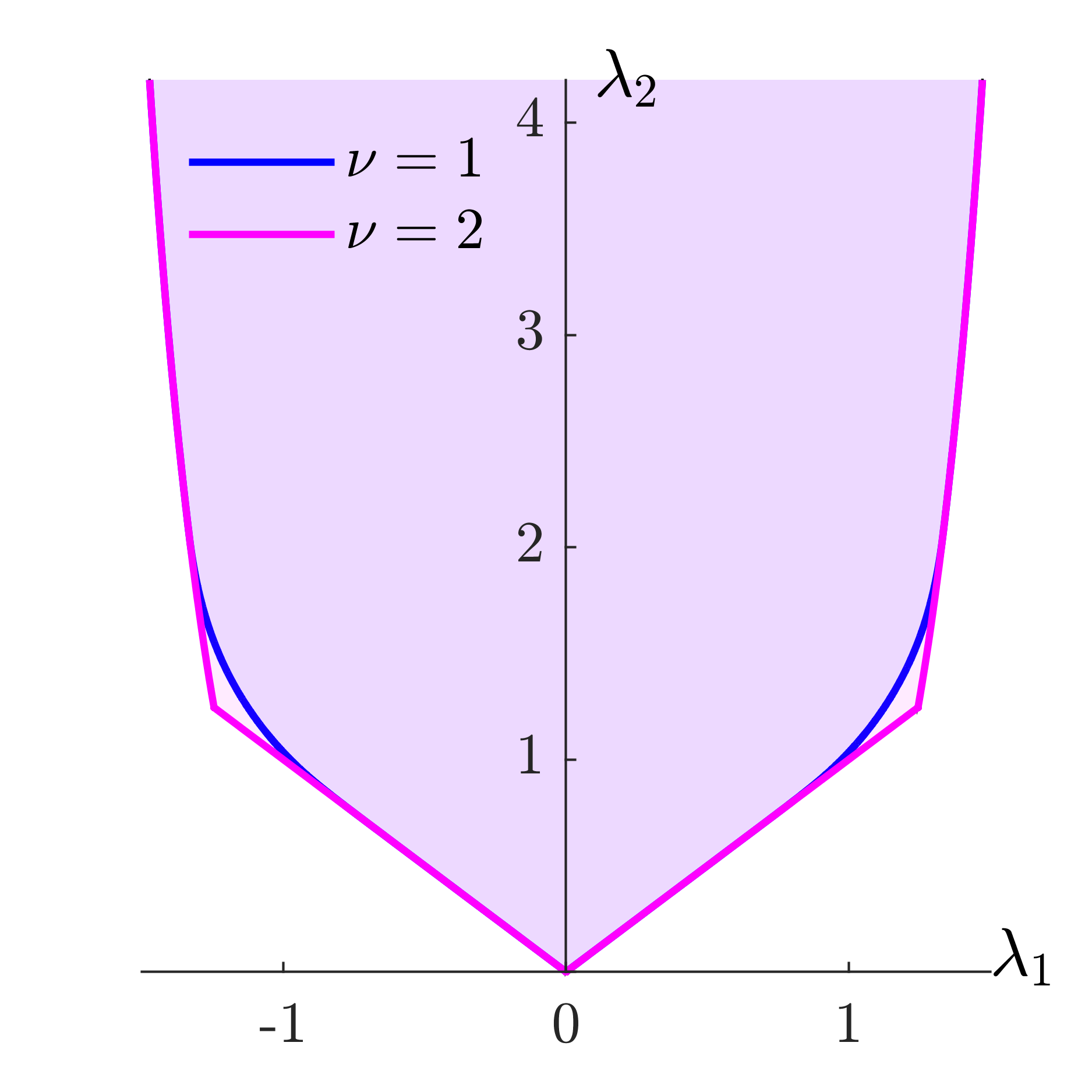}
    }
    \hspace{10mm}
    \subfigure[]{
    \includegraphics[scale = 0.45]{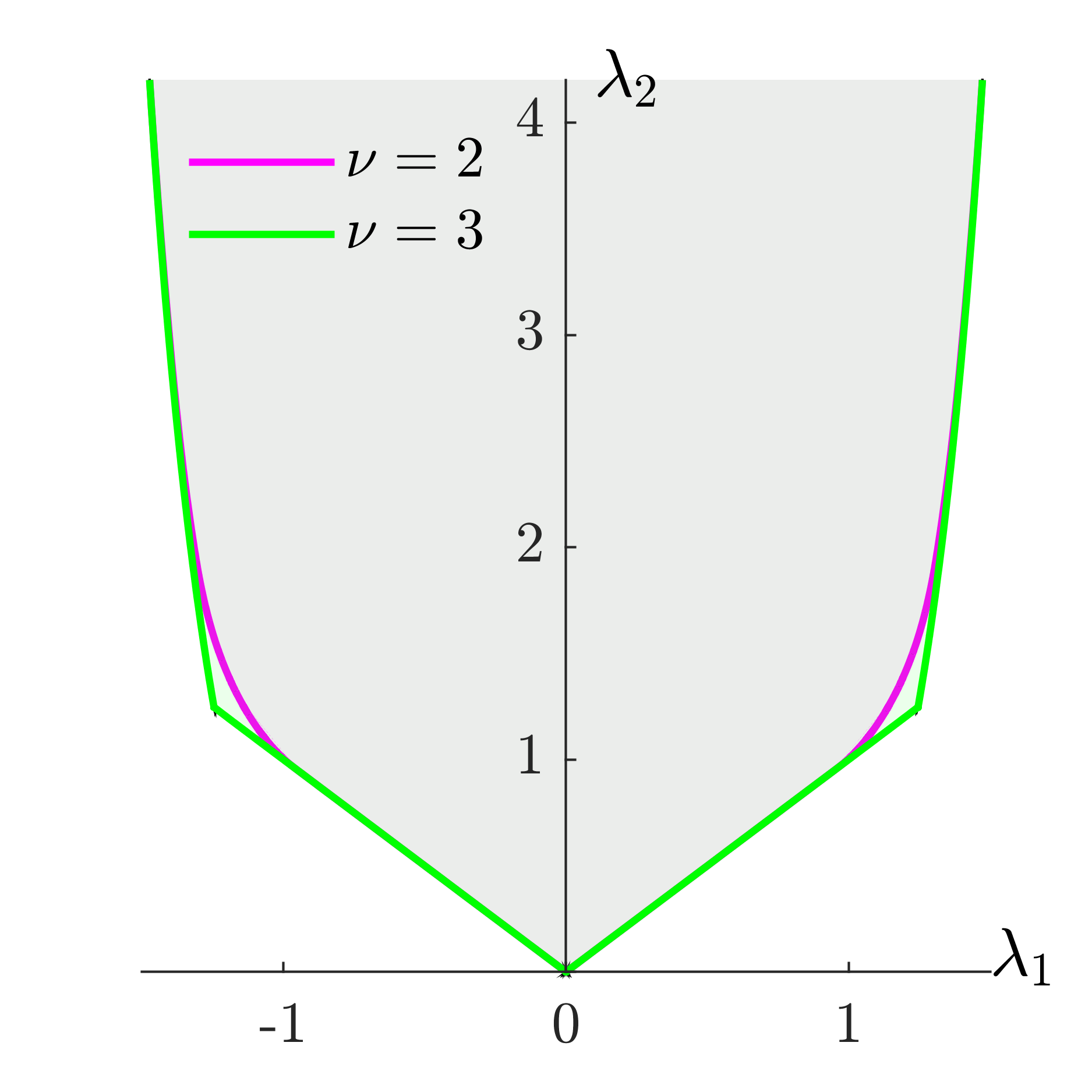}
    }
    \caption{Inner approximations of the set $\mathcal{F}_2$ obtained with SOS optimization. (a) Sets $\mathcal{D}_{2,\nu}$ obtained using the standard SOS constraint~\cref{eq:ex_feasible_region_sos}; (b) Sets $\mathcal{S}_{2,\nu}$ obtained using the sparse SOS constraint~\cref{eq:ex_feasible_region_sos_decomposition}. \Cref{th:sos-program-convergence-homog} guarantees the sequences of sets $\{\mathcal{D}_{2,\nu}\}_{\nu \in \mathbb{N}}$ and $\{\mathcal{S}_{2,\nu}\}_{\nu\in\mathbb{N}}$ are asymptotically exact as $\nu \rightarrow \infty$. The numerical results suggest $\mathcal{S}_{2,3} = \mathcal{D}_{2,2} = \mathcal{F}_2$.}
    \label{fig:ex_feasible_region}
\end{figure}

Next, to illustrate the computational advantages of our sparsity-exploiting SOS methods compared to the standard ones, we use both approaches to bound
\begin{equation} \label{ex:example3_5}
    B^* := \inf_{\lambda \in \mathcal{F}_\omega} \lambda_2 - 10\lambda_1
\end{equation}
from above by replacing $\mathcal{F}_\omega$ with its inner approximations $\mathcal{D}_{\omega,\nu}$ and $\mathcal{S}_{\omega,\nu}$ in~\cref{eq:ex_feasible_region_sos,eq:ex_feasible_region_sos_decomposition}. Optimizing over $\mathcal{D}_{\omega,\nu}$ requires one SOS constraint on a $3\omega \times 3\omega$ polynomial matrix of degree $d=2\nu+4$, while optimizing over $\mathcal{S}_{\omega,\nu}$ requires $3\omega-1$ SOS constraints on $2\times 2$ polynomial matrices of the same degree. \Cref{th:sos-program-convergence-homog} and the inclusion $\mathcal{S}_{\omega,\nu} \subseteq \mathcal{D}_{\omega,\nu}$ guarantee that the upper bounds $B_{d,\nu}$ on $B^*$ obtained with either SOS formulation converge to the latter as $\nu \to \infty$. (Here, as in \cref{ss:applications-to-PMI-optimiz}, $B_{d,\nu}$ denotes the upper bound on $B^*$ obtained from SOS reformulations of~\cref{ex:example3_5} with SOS matrices of degree $d$ and exponent $\nu$.)

\begin{table*}[t]
  \centering
  \renewcommand\arraystretch{1.05}
  \caption{
  Upper bounds $B_{d,\nu}$ on the optimal value $B^*$ of~\cref{ex:example3_5} for increasing values of matrix sizes $\omega$, obtained using the standard SOS constraint~\cref{e:matrix-sos} and the sparsity-exploiting SOS condition~\cref{e:sos-reformulation-general} with SOS matrices of degree $d=4+2\nu$. Also tabulated is the CPU time ($t$, in seconds) required by MOSEK to solve the SDP corresponding to each SOS problem. Entries marked by {\sc oom} indicate ``out of memory'' runtime errors in MOSEK.}
   {\small
	\begin{tabular}{c c cc@{}p{0.3cm}@{}cc@{}p{0.3cm}@{}cc c cc@{}p{0.3cm}@{}cc@{}p{0.3cm}@{}cc}
    \toprule
    && \multicolumn{8}{c}{Standard SOS~\cref{e:matrix-sos}}   && \multicolumn{8}{c}{Sparse SOS~\cref{e:sos-reformulation-general}} \\
    \cline{3-10}\cline{12-19}
    &&
    \multicolumn{2}{c}{$\nu=1$} && \multicolumn{2}{c}{$\nu=2$} && \multicolumn{2}{c}{$\nu=3$} &&
    \multicolumn{2}{c}{$\nu=2$} && \multicolumn{2}{c}{$\nu=3$} && \multicolumn{2}{c}{$\nu=4$} \\
    \cline{3-4}\cline{6-7}\cline{9-10}\cline{12-13}\cline{15-16}\cline{18-19}
    $\omega$ && $t$ & $B_{d,\nu}$ && $t$ & $B_{d,\nu}$ && $t$ & $B_{d,\nu}$ && $t$ & $B_{d,\nu}$ && $t$ & $B_{d,\nu}$ && $t$ & $B_{d,\nu}$ \\[0.25ex]
    5 &&  12 & -8.68 && 25 & -9.36 && 69 & -9.36 &&  0.58 & -8.97 && 0.72 &  -9.36 && 1.29 & -9.36\\
    10 && 407 & -8.33 &&  886 &-9.09 && 2910 & -9.09 &&  1.65 & -8.72 && 0.82 &-9.09 && 2.08 & -9.09\\ 
    15 &&  2090 & -8.26 && {\sc oom}& {\sc oom}&& {\sc oom}& {\sc oom}&&   2.76 & -8.68 && 1.13 &-9.04 && 2.79 & -9.04\\  
    20 &&{\sc oom} & {\sc oom} && {\sc oom}& {\sc oom} && {\sc oom}& {\sc oom}&& 3.24 & -8.66 && 1.54 & -9.02 && 4.70 & -9.02\\
    25 &&{\sc oom} & {\sc oom} && {\sc oom}& {\sc oom} && {\sc oom}& {\sc oom}&& 2.85 & -8.66 && 1.94 & -9.02 && 4.59 & -9.02\\
    30 &&{\sc oom} & {\sc oom} && {\sc oom}& {\sc oom} && {\sc oom}& {\sc oom}&& 2.38 & -8.65 && 2.40 & -9.01 && 5.50 & -9.01\\
    35 &&{\sc oom} & {\sc oom} && {\sc oom}& {\sc oom} && {\sc oom}& {\sc oom}&& 2.66 & -8.65 && 3.25 & -9.01 && 6.17 & -9.01\\
    40 &&{\sc oom} & {\sc oom} && {\sc oom}& {\sc oom} && {\sc oom}& {\sc oom}&& 3.07 & -8.65 && 3.14 & -9.01 && 8.48 & -9.01\\
    \bottomrule
    \end{tabular}
    }
\label{tab:Unconstrained_comparision}
\end{table*}

\Cref{tab:Unconstrained_comparision} lists upper bounds $B_{d,\nu}$ computed with MOSEK using both SOS formulations, degree $d=4+2\nu$, and different values of $\omega$ and $\nu$. The CPU time is also listed. Bounds for our sparse SOS formulation with $\nu=1$ are not reported because MOSEK encountered severe numerical problems irrespective of the matrix size $\omega$.
It is evident that our sparsity-exploiting SOS method scales significantly better than the standard approach as $\omega$ and $\nu$ increase. For $\omega=10$, for example, the bound obtained with our sparsity-exploiting approach and $\nu=3$ agrees to two decimal places with the bounds calculated using traditional methods with $\nu=2$ and $3$, but the computation is three orders of magnitude faster.
More generally, our sparsity-exploiting computations took less than 10 seconds for all tested values of $\omega$ and $\nu$,\footnote{Computations are sometimes faster for $\nu=3$ than for $\nu=2$ because 
MOSEK converged in fewer iterations. This suggests that numerical conditioning improves with $\nu$ for this example.}
while traditional ones required more RAM than available for all but the smallest values. We expect similarly large efficiency gains for any optimization problem with sparse polynomial matrix inequalities if the size of the largest maximal clique of the sparsity graph is much smaller than the matrix size.

\subsection{Approximation of polynomial matrix inequalities on compact sets}

As our second example, we consider the problem of constructing inner approximations for compact sets where a polynomial matrix is positive semidefinite. This problem arises, for instance, when approximating the robust stability region of linear dynamical systems~\cite{scherer2006lmi}, and was studied in~\cite{henrion2011inner} using standard SOS methods. Here, we show that our sparse-matrix version of Putinar's  Positivstellensatz in~\cref{th:sparse-putinar} allows for significant reductions in computational complexity without sacrificing the rigorous convergence guarantees established in~\cite{henrion2011inner}.


Let $\mathcal{K}\subset \mathbb{R}^n$ be a compact semialgebraic set defined as in~\cref{e:semialgebgraic-set} that satisfies the Archimedean condition, and let $P(x)$ be an $m\times m$ symmetric polynomial matrix. We seek to construct a sequence $\{\mathcal{S}_{2d}\}_{d \in \mathbb{N}}$ of subsets of the (compact) set
$\mathcal{P} = \{x \in \mathcal{K} \mid P(x) \succeq 0\}$, 
such that $\mathcal{S}_{2d}$ converges to $\mathcal{P}$ in volume. Following~\cite{henrion2011inner}, this can be done by letting 
$\mathcal{S}_{2d} = \{x \in \mathcal{K} \mid s_{2d}(x) \geq 0\}$ 
be the superlevel set of the degree-$2d$ polynomial $s_{2d}(x)$ that solves the convex optimization problem
\begin{equation} \label{eq:PMI}
            \begin{aligned}
                    B_{m,d}^* := \max_{s_{2d}(x)} \int_{\mathcal{K}} s_{2d}(x) \,{\rm d}x 
                    \quad\text{s.t.}\quad
                    P(x) - s_{2d}(x)I \succeq 0 \quad \forall  x \in \mathcal{K}.
            \end{aligned}
\end{equation}
This problem is in the form~\cref{e:infinite-sdp-intro}, and the optimization variable $\lambda$ is the vector of $\binom{n+2d}{n}$ coefficients of $s_{2d}$ (with respect to any chosen basis). The polynomial $s_{2d}$ is a pointwise lower bound for the minimum eigenvalue function of $P(x)$ on $\mathcal{K}$. Using this observation, the compactness of $\mathcal{K}$, the continuity of eigenvalues, and the Weierstrass polynomial approximation theorem, one can show that, as $d \to \infty$, $\mathcal{S}_{2d}$ converges to $\mathcal{P}$ in volume, $s_{2d}$ converges pointwise almost everywhere to the minimum eigenvalue function, and $B_{m,d}^*$ tends to the integral of the latter on $\mathcal{K}$.

Theorem 1 in~\cite{henrion2011inner} shows that convergence is maintained if the intractable matrix inequality constraint is replaced with a weighted SOS representation for $P(x)-s_{2d}(x)I$ in the form~\cref{e:matrix-sos}, where the SOS matrices $S_k$ are chosen such that the degree of $S_0 + g_1 S_1 + \cdots + g_q S_q$ does not exceed $2d$. By \cref{th:sparse-putinar}, the same is true for the sparsity-exploiting reformulation~\cref{e:sos-reformulation-general} with $\nu=0$, SOS matrices $S_{0,k}$ of degree $d_0 = d$, and SOS matrices $S_{j,k}$ of degree $d_j = d - \lceil\frac12 \deg(g_j) \rceil$.

To illustrate the computational advantages gained by exploiting sparsity, we consider a relatively simple (but still nontrivial) bivariate problem with $\mathcal{K}=\{x \in \mathbb{R}^2: 1-x_1^2 - x_2^2 \geq 0\}$ being the unit disk and
\begin{equation} \label{eq:large_PMI}
    P(x) = (1 - x_1^2 - x_2^2)I_m + (x_1 + x_1x_2 - x_1^3)A + (2x_1^2x_2-x_1x_2-2x_2^3)B,
\end{equation}
where $A$ and $ B$ are $m \times m$ symmetric matrices with chordal sparsity graphs, zero diagonal elements, and other entries drawn randomly from the uniform distribution on $(0,1)$. The sparsity graphs of $A$ and $B$ were generated randomly whilst ensuring that their maximal cliques contain no more than five vertices~\cite{mason2015chordal}, and the corresponding structure of $P$ for $m = 15$, $20$, $25$, $30$, $35$ and $40$ is shown in \cref{fig:ex_sparsity_patterns}. The exact data matrices used in our calculations are available at {\href{https://github.com/aeroimperial-optimization/sos-chordal-decomposition-pmi}{https://github.com/aeroimperial-optimization/sos-chordal-decomposition-pmi}}.

\begin{figure}
    \centering
    \subfigure[$m=15$]{
    \includegraphics[scale = 0.28]{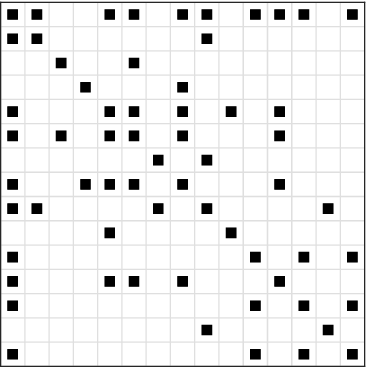}
    }
    \subfigure[$m=20$]{
    \includegraphics[scale = 0.28]{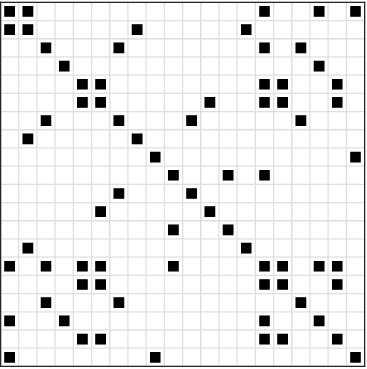}
    }
    \subfigure[$m=25$]{
    \includegraphics[scale = 0.28]{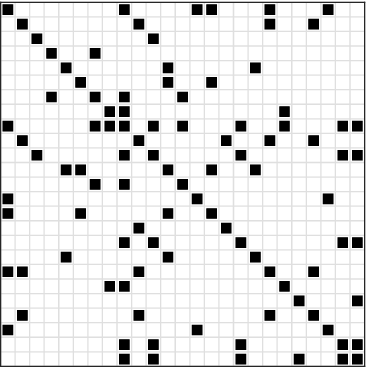}
    } 
    \subfigure[$m=30$]{
    \includegraphics[scale = 0.28]{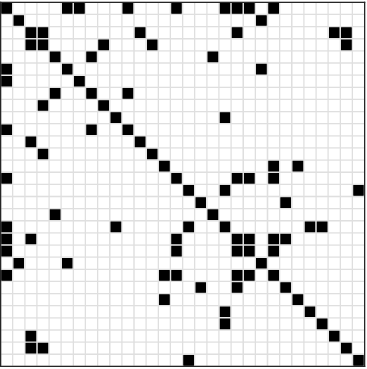}
    }
    \subfigure[$m=35$]{
    \includegraphics[scale = 0.28]{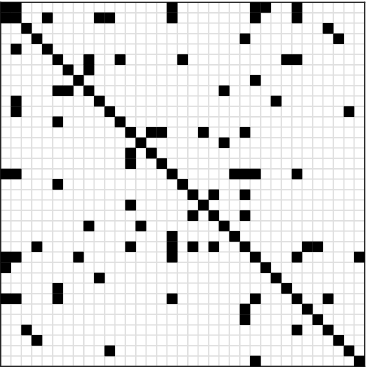}
    }
    \subfigure[$m=40$]{
    \includegraphics[scale = 0.28]{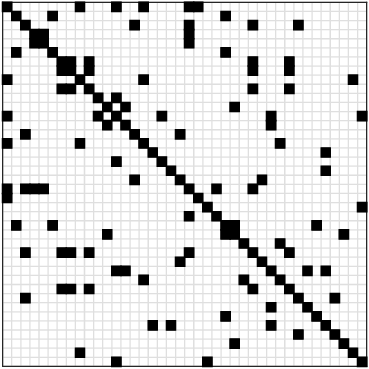}
    }
    \caption{Chordal sparsity patterns for the polynomial matrix $P(x)$ in~\eqref{eq:large_PMI}. 
    }
    \label{fig:ex_sparsity_patterns}
\end{figure}
\begin{figure}
    \includegraphics[width=\textwidth]{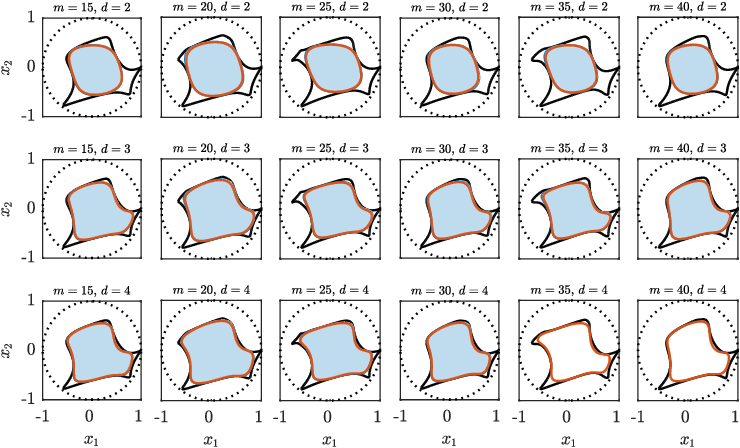}
    \caption{Inner approximations $\mathcal{S}_{2d}$ of the subset $\mathcal{P}$ of the unit disk (black dots) where the sparse $m \times m$ polynomial matrix $P(x)$ in~\cref{eq:large_PMI} is PSD. The boundary of $\mathcal{P}$ is plotted as a solid black line ({\solidrule}).
    The approximating sets computed using the standard SOS constraint~\cref{e:matrix-sos} (blue solid boundary and blue shading; shown if available) and the sparsity-exploiting SOS problem~\cref{e:sos-reformulation-general} with $\nu=0$ (red solid boundary, no shading) 
    and are almost indistinguishable.
   }
    \label{fig:inner_approximations}
\end{figure}

\begin{table*}
  \centering
  \renewcommand\arraystretch{1.05}
  \caption{Lower bounds on the optimal value of~\cref{eq:PMI} with $P(x)$ as in~\cref{eq:large_PMI} and $\mathcal{K}$ the unit disk, obtained using the standard SOS constraint~\cref{e:matrix-sos} and the sparsity-exploiting SOS problem~\cref{e:sos-reformulation-general} for increasing values of $m$ and $d$. Also tabulated is the CPU time ($t$, in seconds) required by MOSEK to solve the SDP corresponding to each SOS problem. Entries marked by {\sc oom} indicate ``out of memory'' runtime errors in MOSEK. The asymptotic value $B_{m,\infty}^*$ was found by integrating the minimum eigenvalue function of $P$ over the unit disk $\mathcal{K}$.}
   {\small
	\begin{tabular}{C{0.2cm}
	        @{}p{0.35cm}@{} C{0.3cm}C{0.9cm}@{}p{0.3cm}@{}C{0.3cm}C{0.9cm}@{}p{0.3cm}@{}C{0.3cm}C{0.9cm}  
			@{}p{0.35cm}@{} C{0.3cm}C{0.9cm}@{}p{0.3cm}@{}C{0.3cm}C{0.9cm}@{}p{0.3cm}@{}C{0.3cm}C{0.9cm}@{}p{0.cm}C{0.8cm}}
    \toprule
    && \multicolumn{8}{c}{Standard SOS~\cref{e:matrix-sos}}   && \multicolumn{8}{c}{Sparse SOS~\cref{e:sos-reformulation-general}} & \\[0.25ex]
    \cline{3-10}\cline{12-19}\\[-2ex]
    &&
    \multicolumn{2}{c}{$d=2$} && \multicolumn{2}{c}{$d=3$} && \multicolumn{2}{c}{$d=4$} &&
    \multicolumn{2}{c}{$d=2$} && \multicolumn{2}{c}{$d=3$} && \multicolumn{2}{c}{$d=4$} && \\
    \cline{3-4}\cline{6-7}\cline{9-10}\cline{12-13}\cline{15-16}\cline{18-19}\\[-2ex]
    $m$ && $t$ & $B^{\rm sos}_{m,d}$ && $t$ & $B^{\rm sos}_{m,d}$ && $t$ & $B^{\rm sos}_{m,d}$ && $t$ & $B^{\rm sos}_{m,d}$ && $t$ & $B^{\rm sos}_{m,d}$ && $t$ & $B^{\rm sos}_{m,d}$ && $B^*_{m,\infty}$\\
     15 && 3.7 & -2.07 &&  24.8 &-1.50 && 95.1 & -1.36 &&  0.95 & -2.10 && 0.97 &-1.52 && 1.94 & -1.37 && -1.15\\ 
    20 &&  13.3 & -1.51 &&  96.5& -1.03&& 375& -0.92&&   0.69 & -1.58 && 1.06 & -1.07 && 2.12 & -0.95 && -0.75\\  
    25 && 38.1 & -2.47 && 326& -1.85 && 1308& -1.64&& 0.95 & -2.50 && 1.28 & -1.87 && 3.04 & -1.66 && -1.41\\
    30 && 136 & -2.13 && 963& -1.54 && 4031& -1.41&& 0.75 & -2.21 && 1.35 & -1.58 && 3.14 & -1.43 && -1.21\\
    35 &&  219 & -2.46 && 2210& -1.82 && {\sc oom}& {\sc oom}&& 0.77 & -2.51 && 1.51 & -1.84 && 3.01 & -1.65 && -1.40\\
    40 && 550 & -2.22 && 5465& -1.59 && {\sc oom}& {\sc oom}&& 1.03 & -2.24 && 2.07 & -1.59 && 5.62 & -1.47 && -1.25\\
    \bottomrule
    \end{tabular}
    }
  \label{tab:constrained_comparision}
\end{table*}

\Cref{fig:inner_approximations} illustrates the inner approximations $\mathcal{S}_{2d}$ of $\mathcal{P}$ computed using both the standard SOS constraint~\cref{e:matrix-sos} and our sparsity-exploiting formulation~\cref{e:sos-reformulation-general}. \Cref{tab:constrained_comparision} lists the corresponding lower bounds $B^{\rm sos}_{m,d}$ on $B_{m,d}^*$, as well as the CPU time required to solve the SOS programs with MOSEK and the limit $B_{m,\infty}^*$ obtained from numerical integration of the minimum eigenvalue function of $P$ on the unit disk $\mathcal{K}$. Similar to what was observed in \cref{example:quartic_polynomial_matrix}, for fixed $d$ the dense SOS constraints give better bounds than the sparse ones. As expected, however, the sparsity-exploiting formulation requires significantly less time for large $m$, and all problem instances were solved within 10 seconds. In addition, the approximating sets $\mathcal{S}_{2d}$ in~\cref{fig:inner_approximations} provided by both SOS formulations for every combination of $d$ and $m$ are almost indistinguishable. For a given matrix size $m$, therefore, our sparse SOS formulation enables the construction of much better approximations to $\mathcal{P}$ by considering large values of $d$, which are beyond the reach of standard SOS formulations. This is important because, as shown in \cref{fig:slow_convergence_m15,tab:large-d-m-15} for $m=15$, the convergence to the set $\mathcal{P}$ and to the limit $B_{m,\infty}^*$ is slow as $d$ is raised. 

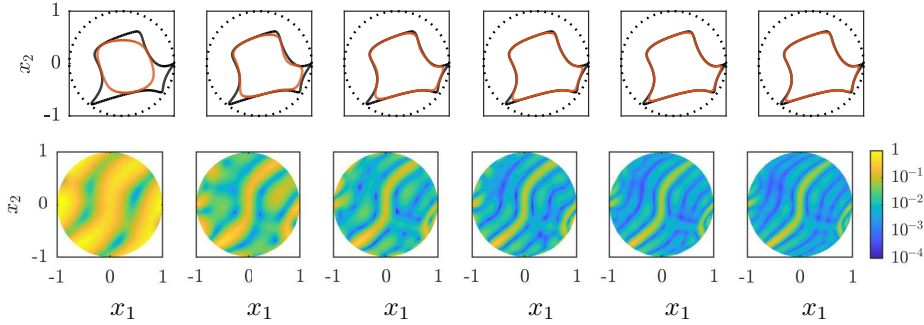
\begin{figure}
    \centering
    \hspace*{-26pt}\includegraphics[scale=0.90]{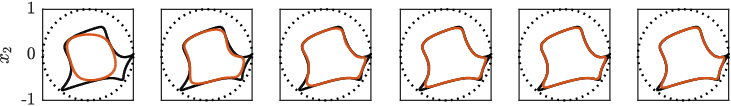}\\[-1pt]
    \hspace*{-30pt}\includegraphics[scale=0.90]{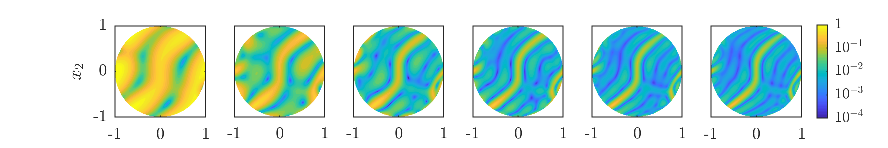}\\
    \hspace*{-10pt}
    \begin{tikzpicture}
    \node at (-4.5,0) {\small $x_1$};
    \node at (-2.7,0) {\small $x_1$};
    \node at (-0.9,0) {\small $x_1$};
    \node at (1,0) {\small $x_1$};
    \node at (2.8,0) {\small $x_1$};
    \node at (4.6,0) {\small $x_1$};
    \end{tikzpicture}
    \caption{\textit{Top:} Boundaries of the set $\mathcal{P}$ (black lines) and of the inner approximations $\mathcal{S}_{2d}$ (red lines) for the matrix $P(x)$ in~\cref{eq:large_PMI} with $m=15$, obtained with the sparse SOS formulation for $d=2$, $4$, $6$, $10$, $12$ and $14$ (left to right). \textit{Bottom:} Absolute difference between the optimal polynomial $s_{2d}$ and the minimum eigenvalue function of $P$ on the unit disk $\mathcal{K}$.}
    \label{fig:slow_convergence_m15}
\end{figure}

\begin{table*}
  \centering
  \caption{Lower bounds $B^{\rm sos}_{15,d}$ on the asymptotic value $B_{15,\infty}^* = -1.153$ of~\cref{eq:PMI} for $m=15$, calculated using the sparsity-exploiting SOS problem~\cref{e:sos-reformulation-general} with $\nu=0$ and the standard SOS constraint~\cref{e:matrix-sos}. The CPU time ($t$, seconds) to compute these bounds using MOSEK is also reported. Entries marked by {\sc oom} indicate ``out of memory'' runtime errors in MOSEK.  
  }
  {\small
	\begin{tabular}{c|cccccc}
    \toprule
    &$d$ 
    & 6 & 8 & 10 & 12 & 14\\ 
    \midrule
   \multirow{2}{*}{Sparse SOS \cref{e:sos-reformulation-general}}  &$B^{\rm sos}_{15,d}$ 
   &  $-1.257$  & {$-1.219$} &  {$-1.199$}  & {$-1.195$} & {$-1.191$}\\
    & $t$ 
    & $\phantom{-0}13.3$ & $\phantom{-0}85.1$ &  $\phantom{-}309.3$ & $\phantom{-}818.3$ & $\phantom{-.}2149$\\
    \midrule
  \multirow{2}{*}{Standard SOS \cref{e:matrix-sos}}  &$B^{\rm sos}_{15,d}$ 
  &  $-1.252$  & {$-1.216$} &  {\sc oom}   & {\sc oom} & {\sc oom}\\
    & $t$ 
    & $\phantom{-.}1133$ & $\phantom{-.}8250$ &  {\sc oom} & {\sc oom} & {\sc oom}\\
    \bottomrule
    \end{tabular}
    }
  \label{tab:large-d-m-15}
\end{table*}

%% file: ss-proof-nonexistence.tex
\subsection{Proof of \texorpdfstring{\cref{th:failure-basic-decomposition}}{Proposition~\ref{th:failure-basic-decomposition}}}
\label{ss:counterexample}

We construct polynomial matrices that cannot be decomposed according to~\cref{e:basic-chordal-decomposition} with polynomial $S_k$. To do so, we may assume that $n=1$ without loss of generality because univariate polynomial matrices are particular cases of multivariate ones.

First, fix $m=3$ and let $\mathcal{G}$ be the sparsity graph of the $3\times 3$ positive definite polynomial matrix considered in \cref{ex:nondecomposable-3d-matrix} for $k=1$,
\begin{equation*}
    P(x) = I_3 + \begin{bmatrix}
             1+x^2 & x+x^2 & 0\\
             x+x^2 & 2x^2  & x-x^2\\
             0     & x-x^2 & x^2
    \end{bmatrix}.
\end{equation*}
Observe that $\mathcal{G}$ is essentially the only connected but not complete graph with $m=3$: any other such graph can be reduced to $\mathcal{G}$ by reordering its vertices, which corresponds to a symmetric permutation of the polynomial matrix it describes. We have already shown in \cref{ex:nondecomposable-3d-matrix} that $P$ has no decomposition of the form~\cref{e:basic-chordal-decomposition} with polynomial $S_k$, so \cref{th:failure-basic-decomposition} holds for $m=3$.

The same $3\times 3$ matrix can be used to generate counterexamples for a general connected but not complete sparsity graph $\mathcal{G}$ with $m > 3$. 
Non-completeness implies that $\mathcal{G}$ must have at least two maximal cliques, while connectedness implies that every maximal clique $\mathcal{C}_i$ must contain at least two elements and intersect at least one other clique $\mathcal{C}_j$. Whenever $\mathcal{C}_i\cap \mathcal{C}_j \neq \emptyset$, therefore, there exist vertices $v_i \in \mathcal{C}_i \setminus \mathcal{C}_j$, $v_j \in \mathcal{C}_j \setminus \mathcal{C}_i$ and $v_k \in \mathcal{C}_i\cap \mathcal{C}_j$. Moreover, since $\mathcal{G}$ is chordal, Theorem~3.3 in~\cite{vandenberghe2015chordal} guarantees that it contains at least one simplicial vertex (cf. \cref{s:preliminaries} for a definition), which must belong to one and only one maximal clique. Upon reordering the vertices and the maximal cliques if necessary, we may therefore assume without loss of generality that: (i) $\mathcal{C}_1 = \{1,\ldots,r\}$ for some $r$; (ii) vertex $1$ is simplicial, so it belongs only to clique $\mathcal{C}_1$; (iii) vertex $2$ is in $\mathcal{C}_1 \cap \mathcal{C}_2$ and vertex $r+1$ is in $\mathcal{C}_2 \setminus \mathcal{C}_1$. 

Now, consider the positive definite $m \times m$ matrix
\begin{equation*}
    P(x) 
    =
    I_m + 
    E_{\{1,2,r+1\}}^\tr  
    \begin{bmatrix}
    1+x^2 & x+x^2 & 0\\
    x+x^2 & 2x^2  & x-x^2\\
    0     & x-x^2 & x^2
    \end{bmatrix}
    E_{\{1,2,r+1\}},
\end{equation*}
whose nonzero entries are on the diagonal or in the principal submatrix with rows and columns indexed by $\{1,2,r+1\}$. Note that the sparsity pattern of $P$ is compatible with the sparsity graph $\mathcal{G}$. We claim that no decomposition of the form~\cref{e:basic-chordal-decomposition} exists where each $S_k$ is a PSD polynomial matrix. 

For the sake of contradiction, assume that such a decomposition exists, so
\begin{equation*}
    P(x) = 
    E_{\mathcal{C}_1}^\tr S_1(x) E_{\mathcal{C}_1} + \sum_{k=2}^t E_{\mathcal{C}_k}^\tr S_k(x) E_{\mathcal{C}_k} =: E_{\mathcal{C}_1}^\tr S_1(x) E_{\mathcal{C}_1} + Q(x),
\end{equation*}
where $S_1(x)$ and $Q(x)$ are $r\times r$ and $m \times m$ PSD polynomial matrices, respectively. Since vertex $1$ is contained only in clique $\mathcal{C}_1$, the matrix $S_1$ must have the form
\begin{equation*}
    S_1(x) = \begin{bmatrix}
    2+x^2 & (x+x^2,\, 0,\,\ldots,\,0)\\
    (x+x^2,\, 0,\,\ldots,\,0)^\tr & T(x)
    \end{bmatrix}
\end{equation*}
for some $(r-1)\times(r-1)$ polynomial matrix $T$ to be determined.
For the same reason, the matrix $Q(x)$ can be partitioned as
\begin{equation*}
    Q(x) = \begin{bmatrix}
    0 & 0_{1 \times (r-1)} & 0_{1\times (m-r)}\\[0.5ex]
    0_{(r-1) \times 1} & A(x) & B(x) \\ 
     0_{(m-r) \times 1} &  B(x)^\tr & C(x)
    \end{bmatrix},
\end{equation*}
where $0_{p\times q}$ is a $p\times q$ matrix of zeros, $A$ is an $(r-1)\times (r-1)$ polynomial matrix to be determined, and the $(r-1) \times (m-r)$ block $B$ and the $(m-r) \times (m-r)$ block $C$ are given by
\begin{align*}
	B(x) &=  \left[\begin{array}{cl}
	x-x^2 & 0_{1 \times (m-r-1)}\\
	0_{(r-2) \times 1} & 0_{(r-2) \times (m-r-1)}
	\end{array}\right],
	&
	C(x) &= \begin{bmatrix}
	x^2+2 & 0_{1 \times (m-r-1)}\\
	0_{(m-r-1) \times 1} & I_{m-r-1}
	\end{bmatrix}.
\end{align*}

The block $T$ of $S_1$ and the block $A$ of $Q$ correspond to element of clique $\mathcal{C}_1$ that may belong also to other cliques. These blocks cannot be determined uniquely, but their sum must be equal to the principal submatrix of $P$ with rows and columns indexed by $\{2,\ldots,r\}$. In particular, we must have $A_{11}(x) = 2x^2+1 - T_{11}(x)$. Moreover, since $S_1$ and $Q$ are PSD by assumption, we may take appropriate Schur complements to find
\begin{align*}
    T(x) &\succeq  
	\left[\begin{array}{cl}
	\frac{x^2(1+x)^2}{x^2+2} & 0_{1 \times (r-2)}\\
	0_{(r-2) \times 1} & 0_{(r-2) \times (r-2)}
	\end{array}\right],
    &
    A(x) &\succeq
	\left[\begin{array}{cl}
	\frac{x^2(1-x)^2}{x^2+2} & 0_{1 \times (r-2)}\\
	0_{(r-2) \times 1} & 0_{(r-2) \times (r-2)}
	\end{array}\right].
\end{align*}
Using the identity $A_{11}(x) = 2x^2+1 - T_{11}(x)$, these conditions require
\begin{equation*}
    T_{11}(x) \geq \frac{x^2(1+x)^2}{x^2+2}, \qquad 
    2 x^2+1 - T_{11}(x) \geq \frac{x^2(1-x)^2}{x^2+2}.
\end{equation*}
However, just as in \cref{ex:nondecomposable-3d-matrix}, no polynomial $T_{11}(x)$ can satisfy these inequalities. We conclude that $P$ cannot admit a decomposition of the form~\cref{e:basic-chordal-decomposition} with PSD polynomial matrices $S_k$, which proves \cref{th:failure-basic-decomposition} in the general case.

%% file: ss-proof-hilbert-artin.tex
\subsection{Proof of \texorpdfstring{\cref{th:weighted-chordal-decomposition}}{Theorem~\ref{th:weighted-chordal-decomposition}}}
\label{ss:proof-weighted-chordal-decomposition}

To establish \cref{th:weighted-chordal-decomposition} we adapt ideas by Kakimura~\cite{kakimura2010direct}, who proved the chordal decomposition theorem for constant PSD matrices (cf. \cref{th:ChordalDecompositionTheorem}) using the fact that symmetric matrices with chordal sparsity patterns admit an $LDL^\tr$ factorization with no fill-in~\cite{rose1970triangulated}. In \cref{app:diagonalization}, we use Schm\"udgen's diagonalization procedure~\cite{Schmudgen2009noncommutative} to prove the following analogous statement for polynomial matrices. 

\begin{proposition}
\label{Prop:diagonalization}
     	If $P(x)$ is an $m\times m$ symmetric polynomial matrix with chordal sparsity graph, there exist an $m \times m$ permutation matrix $T$, 
     	an invertible $m \times m$ lower-triangular polynomial matrix $L(x)$, and polynomials $b(x)$, $d_1(x),\,\ldots,\,d_m(x)$ such that 
         \begin{equation} \label{eq:diagonalization}
            b^4(x)\, T P(x) T^\tr = L(x) \Diag\left( d_1(x),\,\ldots,\,d_m(x) \right) L(x)^\tr.
         \end{equation}
         Moreover, $L$ has no fill-in in the sense that $L + L^\tr$ has the same sparsity as $TPT^\tr$. 
\end{proposition}



Now, let $P(x)$ be a PSD polynomial matrix with chordal sparsity graph, and apply \cref{Prop:diagonalization} to diagonalize it. We will assume first that the permutation matrix $T$ is the identity, and remove this assumption at the end. 

Since $P$ is PSD, the polynomials $d_1(x),\,\ldots,\,d_m(x)$ in~\cref{eq:diagonalization} must be nonnegative globally and, by the Hilbert--Artin theorem~\cite{artin1927zerlegung}, can be written as sum of squares of rational functions. In particular, there exist SOS polynomials $f_1,\,\ldots,\,f_m$ and $g_1,\,\ldots,\,g_m$ such that 
$f_i(x) d_i(x) = g_i(x)$ for all $i=1,\,\ldots,\,m$. Therefore, we can write (omitting the argument $x$ for notational simplicity)
\begin{equation*}
    \prod_{j=1}^m f_j b^4 \, P = L 
    \Diag\bigg( g_1\prod_{j\neq 1} f_j ,\,\ldots,\,g_i\prod_{j\neq i} f_j,\,\ldots,\, g_m\prod_{j\neq m} f_j  \bigg)
    L^\tr.
\end{equation*}

Next, define the polynomial $\sigma := \prod_{j} f_j b^4$ and observe that it SOS because it is the product of SOS polynomials. For the same reason, the products $ g_i\prod_{j\neq i} f_j $  appearing on the right-hand side of the last equation are SOS polynomials. Thus, we can find an integer $s$ and polynomials $q_{11},\,\ldots,\,q_{m1},\,\ldots,\,q_{1s},\,\ldots,\,q_{ms}$ such that
\begin{equation}
\sigma P = 
\sum_{i=1}^s L 
\Diag\left(q_{1i}^2,\,\ldots,\,q_{mi}^2\right)
L^\tr
=: \sum_{i=1}^s H_i H_i^\tr,
\label{eq:HH}
\end{equation}
where, for notational simplicity,  we have introduced the lower-triangular matrices $$H_i := L \Diag\left(q_{1i},\,\ldots,\,q_{mi}\right).$$

Under our additional assumption that \cref{Prop:diagonalization} can be applied with $T=I$, \cref{th:weighted-chordal-decomposition} follows if we can show that 
\begin{equation} \label{eq:HiHi}
    H_i H_i^\tr = \sum_{k=1}^t E_{\mathcal{C}_k}^\tr S_{ik} E_{\mathcal{C}_k}
\end{equation}
for some SOS matrices $S_{ik}$ and each $i=1,\,\ldots,\,s$. Indeed, combining \cref{eq:HiHi} with \cref{eq:HH} and setting $S_k = \sum_{i=1}^s S_{ik}$ yields the desired decomposition~\cref{E:DecompositionSparseCone} for $P$. 

To establish~\cref{eq:HiHi}, denote the columns of $H_i$ by $h_{i1},\ldots,h_{im}$ and write
\begin{equation}\label{eq:HiHitranspose}
    H_iH_i^\tr = \sum_{j =1}^m h_{ij}h_{ij}^\tr.
\end{equation}
Since $H_i$ has the same sparsity pattern as $L$, the nonzero elements of each column vector $h_{ij}$ must be indexed by a clique $\mathcal{C}_{\ell_j}$ for some $\ell_j \in \{1, \ldots, t\}$. Thus, the nonzero elements of $h_{ij}$ can be extracted through multiplication by the matrix $E_{\mathcal{C}_{\ell_j}}$ and $h_{ij} = E_{\mathcal{C}_{\ell_j}}^\tr E_{\mathcal{C}_{\ell_j}}h_{ij}$.
Consequently,
\begin{equation} \label{eq:hij_clique}
    h_{ij}h_{ij}^\tr =  E_{\mathcal{C}_{\ell_j}}^\tr 
    \underbrace{\left(E_{\mathcal{C}_{\ell_j}}h_{ij}h_{ij}^\tr E_{\mathcal{C}_{\ell_j}}^\tr\right)}_{=:Q_{ij}}
    E_{\mathcal{C}_{\ell_j}}
\end{equation}
where $Q_{ij}$ is an SOS matrix by construction. Now, let $J_{ik} = \{j: \ell_j = k\}$ be the set of column indices $j$ such that column $h_{ij}$ is indexed by clique $\mathcal{C}_k$. These index sets are disjoint and $\cup_k J_{ik} = \{1,\ldots,m\}$, so substituting~\cref{eq:hij_clique} into~\cref{eq:HiHitranspose} we obtain
\begin{equation*}
    H_iH_i^\tr 
    = \sum_{j =1}^m E_{\mathcal{C}_{\ell_j}}^\tr Q_{ij} E_{\mathcal{C}_{\ell_j}}
    = \sum_{k=1}^t \sum_{j \in J_{ik}} E_{\mathcal{C}_k}^\tr Q_{ij} E_{\mathcal{C}_k}
    = \sum_{k=1}^t E_{\mathcal{C}_k}^\tr \bigg( \sum_{j \in J_{ik}} Q_{ij} \bigg) E_{\mathcal{C}_k}.
\end{equation*}
This is exactly~\cref{eq:HiHi} with matrices $S_{ik} = \sum_{j \in J_{ik}} Q_{ij}$, which are SOS because they are sums of SOS matrices. Thus, we have proved \cref{th:weighted-chordal-decomposition} for polynomial matrices $P$ to which \cref{Prop:diagonalization} can be applied with $T=I$.

The general case follows from a relatively straightforward permutation argument. First, apply the argument above to decompose the permuted matrix $T P T^\tr$, whose sparsity graph $\mathcal{G}'$ is obtained by reordering the vertices of the sparsity graph $\mathcal{G}$ of $P$ according to the permutation $T$. Second, observe that the cliques $\mathcal{C}_1,\ldots,\mathcal{C}_t$ of $\mathcal{G}$ are related to the cliques $\mathcal{C}'_1,\ldots,\mathcal{C}'_t$ of $\mathcal{G}'$ by the permutation $T$, so the matrices $E_{\mathcal{C}_k}$ and $E_{\mathcal{C}'_k}$ satisfy $E_{\mathcal{C}_k} = E_{\mathcal{C}'_k}T$. 
As required, therefore,
%
\begin{equation*}
    \sigma(x) P(x)
    = T^\tr \!\left[ \sigma(x) T P(x) T^\tr \right]\! T
    = T^\tr \!\left[ \sum_{k=1}^{t}\!  E_{\mathcal{C}'_k}^\tr S_k(x) E_{\mathcal{C}'_k} \!\right]\! T
    = \sum_{k=1}^{t} \!E_{\mathcal{C}_k}^\tr S_k(x) E_{\mathcal{C}_k}.
\end{equation*}
%

%% file: ss-proof-putinar.tex
\subsection{Proof of \texorpdfstring{\cref{th:sparse-putinar}}{Theorem~\ref{th:sparse-putinar}}}
\label{ss:proof-sparse-putinar}

Our proof of \cref{th:sparse-putinar} follows the same steps used by Kakimura~\cite{kakimura2010direct} to prove the chordal decomposition theorem for constant PSD matrices (\cref{th:ChordalDecompositionTheorem}). Borrowing ideas from~\cite{grimm2007note}, this can be done with the help of the Weierstrass polynomial approximation theorem and the following version of Putinar's Positivstellensatz for polynomial matrices due to Scherer and Hol~\cite[Theorem~2]{scherer2006matrix}. 

\begin{theorem}[Scherer and Hol~\cite{scherer2006matrix}] \label{th:matrixPutinar} 
	Let $\mathcal{K}$ be a compact semialgebraic set defined as in~\cref{e:semialgebgraic-set} that satisfies the Archimedean condition~\cref{e:archimedean-condition}. If an $m\times m$ symmetric polynomial matrix $P(x)$ is strictly positive definite on $\mathcal{K}$, there exist $m \times m$ SOS matrices $S_0,\,\ldots,\,S_q$ such that 
	$P(x) = S_0(x) + \sum_{i=1}^q S_i(x)g_i(x)$.
\end{theorem}
\begin{remark}
It is also possible to establish \cref{th:sparse-putinar} by  modifying the proof of \cref{th:matrixPutinar} with the help of~\cref{th:ChordalDecompositionTheorem}. This alternative approach is technically more involved, but might be extended more easily to obtain sparsity-exploiting versions of the general result in~\cite[Corollary 1]{scherer2006matrix}, rather than of its particular version in \cref{th:matrixPutinar}. We leave this generalization to future research.
\end{remark}

Let $P(x)$ be an $m\times m$ polynomial matrix with chordal sparsity graph $\mathcal{G}$. 
If $m=1$ or $2$, \cref{th:sparse-putinar} is a direct consequence of~\cref{th:matrixPutinar}.
For $m\geq 3$, we proceed by induction assuming that \cref{th:sparse-putinar} holds for matrices of size $m-1$ or less. Without loss of generality, we assume that the sparsity graph $\mathcal{G}$ is not complete (otherwise, $P$ is dense and \cref{th:sparse-putinar} reduces to~\cref{th:matrixPutinar}) and connected (otherwise, $P$ and $\mathcal{G}$ can be replaced by their connected components). 

Since $\mathcal{G}$ is chordal, it has at least one simplicial vertex~\cite[Theorem~3.3]{vandenberghe2015chordal}. Relabelling vertices if necessary, which is equivalent to permuting $P$, we may assume that vertex $1$ is simplicial and that the first maximal clique of $\mathcal{G}$ is $\mathcal{C}_1 = \{1,\ldots,r\}$ with  $1< r < m$. Thus, $P(x)$ has the block structure
\begin{equation*}
    P(x) = \begin{bmatrix}
     a(x) & {b}(x)^\tr & 0 \\
     {b}(x) & U(x) & V(x)\\
     0 & V(x) & W(x)
    \end{bmatrix}
\end{equation*}
for some polynomial $a$, polynomial vector $b=(b_1,\,\ldots,\,b_{r-1})$, and polynomial matrices $U$ of dimension $(r-1) \times (r-1)$, $V$ of dimension $(r-1)\times (m-r)$, and $W$ of dimension $(m-r) \times (m-r)$. 

The polynomial $a$ must be strictly positive on $\mathcal{K}$ because $P$ is positive definite on that set, so we can apply one step of the Cholesky factorization algorithm to write
\begin{equation} \label{eq:step_s1}
    L(x) P(x) L(x)^\tr = \begin{bmatrix}
     a(x) & 0 & 0 \\
     0 & U(x) - a(x)^{-1}{b(x)}{b(x)}^\tr & V(x)\\
     0 & V(x)^\tr & W(x)
    \end{bmatrix},
\end{equation}
where 
\begin{equation*}
    L(x) := \begin{bmatrix}
     1 & 0 & 0\\
     -a(x)^{-1}{b(x)} & I & 0\\
     0 & 0 & I
    \end{bmatrix}.
\end{equation*}
The matrix on the right-hand side of~\cref{eq:step_s1} is positive definite on the compact set $\mathcal{K}$ because so is $P$ and $L$ is invertible.
Therefore, 
there exists $\varepsilon>0$ such that
\begin{equation} \label{eq:step2_1}
    \begin{bmatrix}
        U(x) - a(x)^{-1}{b(x)}{b(x)}^\tr & V(x)\\
        V(x)^\tr & W(x)
    \end{bmatrix} \succ 4 \varepsilon I \quad  \forall x \in \mathcal{K}. 
\end{equation}
Moreover, the rational entries of the matrix $a^{-1} b b^\tr$ are continuous on $\mathcal{K}$ because $a$ is strictly positive on that set, so we may apply the Weierstrass approximation theorem to choose a polynomial matrix $H(x)$ that satisfies
\begin{equation} \label{eq:step3}
    \phantom{\qquad \forall x \in \mathcal{K}.}
    -\varepsilon I  \preceq  H(x) - a(x)^{-1}{b(x)}{b(x)}^\tr \preceq \varepsilon I \quad \forall x \in \mathcal{K}.
\end{equation}

Next, consider the decomposition
\begin{equation} \label{eq:step_s4}
    P(x) = \begin{bmatrix}
     a(x) & {b}(x)^\tr & 0 \\
     {b}(x) & H(x) + 2\varepsilon I & 0\\
     0 & 0 & 0
    \end{bmatrix}
    +
    \begin{bmatrix}
     0 & 0 & 0 \\
     0 & U(x) - H(x)-2\varepsilon I & V(x)\\
     0 & V(x)^\tr & W(x)
    \end{bmatrix}.
\end{equation}
Combining~\cref{eq:step3} with {the strict positivity} of $a(x)$ on $\mathcal{K}$ we obtain 
\begin{equation*}
Q(x) 
:=    \begin{bmatrix}a(x) & {b(x)}^\tr \\{b(x)} & H(x)+2\varepsilon I\end{bmatrix}
\succeq \begin{bmatrix}a(x) & {b(x)}^\tr \\{b(x)} & a(x)^{-1}{b(x)}{b(x)}^\tr + \varepsilon I\end{bmatrix}
\succ 0 \qquad \forall x \in \mathcal{K},
\end{equation*}
where the last strict matrix inequality follows from the strict positivity of $a$ and Schur's complement conditions.
Since $Q$ is positive definite on $\mathcal{K}$, we may apply \cref{th:matrixPutinar} to find SOS matrices $T_{0},\,\ldots,\,T_{q}$ such that
\begin{equation}\label{e:Q-sos-decomposition}
    Q(x) = T_{0}(x) + \sum_{i=1}^q g_i(x) T_{i}(x).
\end{equation}
Moreover, for all $x \in \mathcal{K}$ inequalities~\cref{eq:step2_1} and~\cref{eq:step3} yield
\begin{align*}
    R(x)
    := \begin{bmatrix}U - H(x)- 2\varepsilon I & V(x)\\V(x)^\tr & W(x)\end{bmatrix}
    \succeq \begin{bmatrix}U - a(x)^{-1}{b(x)}{b(x)}^\tr- 3\varepsilon I & V(x)\\V(x)^\tr & W(x)\end{bmatrix}
    \succeq \varepsilon I.
\end{align*}
The sparsity of $R(x)$ is described by the subgraph $\tilde{\mathcal{G}}$ of $\mathcal{G}$ obtained by removing the simplicial vertex $1$ and its corresponding edges. This subgraph is chordal~\cite[Section~4.2]{vandenberghe2015chordal} and has either $t$ maximal cliques $\tilde{\mathcal{C}}_1 = \mathcal{C}_1\setminus\{1\},\,\tilde{\mathcal{C}}_2 = \mathcal{C}_2,\,\ldots,\,\tilde{\mathcal{C}}_t=\mathcal{C}_t$, or $t-1$ maximal cliques $\tilde{\mathcal{C}}_2=\mathcal{C}_2,\,\ldots,\,\tilde{\mathcal{C}}_{t}=\mathcal{C}_t$ (in the latter case, we set $\tilde{\mathcal{C}}_1=\emptyset$ for notational convenience). In either case, by the induction hypothesis, we can find  SOS matrices $Y_i$ and $\tilde{S}_{ik}$ such that (omitting the argument $x$ from all polynomials and polynomial matrices for notational simplicity)\footnote{Here we slightly abuse notation: the matrices $E_{\tilde{\mathcal{C}}_k}$ have size $\vert \tilde{\mathcal{C}}_k \vert \times (m-1)$ because they are defined using the graph $\tilde{\mathcal{G}}$, which has $m-1$ vertices. The matrices $E_{\mathcal{C}_k}$, instead, have size $\vert {\mathcal{C}}_k \vert \times m$ because they are defined using the graph $\mathcal{G}$, which has $m$ vertices.}
%
\begin{equation}\label{e:R-sos-decomposition}
R
= 
E_{\tilde{\mathcal{C}}_1}^\tr \bigg( Y_{0} + \sum_{i=1}^q g_i Y_{i} \bigg) E_{\tilde{\mathcal{C}}_1}
+ \sum_{k=2}^t E_{\tilde{\mathcal{C}}_k}^\tr \bigg(\tilde{S}_{0k} + \sum_{i=1}^q g_i \tilde{S}_{ik}\bigg) E_{\tilde{\mathcal{C}}_k}.
\end{equation}

The SOS decomposition~\cref{e:Q-sos-decomposition,e:R-sos-decomposition} can now be combined with~\cref{eq:step_s4} to derive the desired SOS decomposition for $P(x)$. The process is straightforward but cumbersome in notation, because we need to handle matrices of different dimensions. For each $i \in \{0,\ldots,q\}$ and $k \in \{1,\ldots,t\}$ define the matrices
\begin{align*}
&&
    Z_i(x) &:= \begin{bmatrix} 0 & 0 \\ 0 & Y_i(x)\end{bmatrix},
    &
    S_{ik}(x) &:= \begin{bmatrix} 0 & 0 \\ 0 & \tilde{S}_{ik}(x)\end{bmatrix},
&&
\end{align*}
and note that
\begin{align*}
    \begin{bmatrix}
    0\hspace{6pt} & 0 \\ 0\hspace{6pt} & E_{\tilde{\mathcal{C}}_1}^\tr Y_i(x)E_{\tilde{\mathcal{C}}_1}
    \end{bmatrix}
    &=
    E_{\mathcal{C}_1}^\tr Z_i(x) E_{\mathcal{C}_1}.
    &
    \begin{bmatrix}
    0 & 0 \\ 0 & E_{\tilde{\mathcal{C}}_k}^\tr \tilde{S}_{ik}(x)E_{\tilde{\mathcal{C}}_k}
    \end{bmatrix}
    &=
    E_{\mathcal{C}_k}^\tr {S}_{ik}(x) E_{\mathcal{C}_k}.
\end{align*}
We therefore obtain
\begin{equation*}
   \begin{bmatrix}
    0 & 0 \\ 0 & R
    \end{bmatrix} 
    =
    E_{\mathcal{C}_1}^\tr \bigg( Z_{0} + \sum_{i=1}^q g_i Z_{i} \bigg) E_{\mathcal{C}_1}
    + \sum_{k=2}^t E_{\mathcal{C}_k}^\tr \bigg(S_{0k} + \sum_{i=1}^q g_i S_{ik}\bigg) E_{\mathcal{C}_k}
\end{equation*}
and can rewrite the decomposition~\cref{eq:step_s4} as
\begin{equation*}
    P 
    = E_{\mathcal{C}_1}^\tr \bigg( Q + Z_{0} + \sum_{i=1}^q g_i Z_{i} \bigg) E_{\mathcal{C}_1}
    + \sum_{k=2}^t E_{\mathcal{C}_k}^\tr \bigg(S_{0k} + \sum_{i=1}^q g_i S_{ik}\bigg) E_{\mathcal{C}_k}.
\end{equation*}
Substituting the decomposition of $Q$ from~\cref{e:Q-sos-decomposition}, letting $S_{i1}(x) := T_i(x) + Z_i(x)$, and reintroducing the $x$-dependence of various terms we arrive at
\begin{equation*}
    P(x) = \sum_{k=1}^t E_{\mathcal{C}_k}^\tr \bigg(S_{0k}(x) + \sum_{i=1}^q g_i(x) S_{ik}(x)\bigg) E_{\mathcal{C}_k}.
\end{equation*}
which is the desired SOS decomposition of $P(x)$. 

%% file: ss-proof-putinar-vasilescu.tex
\subsection{Proof of \texorpdfstring{\cref{th:sparse-putinar-vasilescu-homog}}{Theorem~\ref{th:sparse-putinar-vasilescu-homog}}}
\label{ss:proof-sparse-putinar-vasilescu}

We combine the argument given in~\cite{Dinh2021} for general (dense) polynomial matrices with \cref{th:sparse-putinar} and the following auxiliary result, proven in \cref{app:sos-symmetric-matrix-lemma}.
\begin{lemma}\label{lemma:sos-symmetric-matrix}
	Let $S(x)$ be an SOS polynomial matrix satisfying $S(x)=S(-x)$. For any real number $r\geq 0$ and any integer $\omega$ such that $2\omega \geq \deg(S)$, the matrix $\|x\|^{2\omega} S(r\|x\|^{-1}x)$ is polynomial of degree $2\omega$, homogeneous, and SOS.
\end{lemma}

Choose any nonzero $x_0 \in \mathcal{K}$, let $r = \|x_0\| \neq 0$, and observe that the (nonempty) semialgebraic set $\mathcal{K}' := \mathcal{K} \cap \{x \in \mathbb{R}^n: \pm(r^2 - \|x\|^2) \geq 0 \}$ satisfies the Archimedean condition~\cref{e:archimedean-condition}. Set $g_{q+1}(x) = r^2 - \|x\|^2$ and $g_{q+2}(x) = \|x\|^2 -r^2$ for notational convenience. Since the homogeneous polynomial matrix $P(x')$ is strictly positive definite for all $x' \in \mathcal{K}' \subseteq \mathcal{K}\setminus \{0\}$, we can apply \cref{th:sparse-putinar} to find SOS matrices $S_{j,k}'$ (not necessarily homogeneous) such that
\begin{equation}\label{e:augmented-putinar-decomposition}
P(x') = 
\sum_{k=1}^t E_{\mathcal{C}_k}^\tr 
	\bigg( \hat{S}_{0,k}(x') + \sum_{j=1}^{q+2} g_j(x')\hat{S}_{j,k}(x') \bigg)
	E_{\mathcal{C}_k}.
\end{equation} 
Moreover, standard symmetry arguments (see, e.g.,~\cite{lofberg2009pre,Riener2013}) reveal that we may take $\hat{S}_{j,k}(-x')=\hat{S}_{j,k}(x')$ for all $j$ and $k$ 
because the matrix $P$ and the polynomials $g_1,\ldots,g_{q+2}$ are invariant under the transformation $x \mapsto -x$. The latter assertion is true because $P$ and $g_1,\ldots,g_{q}$ are homogeneous and have even degree by assumption, while $g_{q+1}(-x')=g_{q+1}(x')$ and $g_{q+2}(-x')=g_{q+2}(x')$ by construction. 

Next, set $2d_0 = \deg(P)$ and $2d_j = \deg(g_j)$ for all $j=1,\ldots,q$. Given any nonzero $x \in \mathbb{R}^n$, evaluating~\cref{e:augmented-putinar-decomposition} at the point $x' = r x \|x\|^{-1}$ yields
\begin{equation}\label{e:augmented-putinar-decomposition-2}
\frac{r^{2d_0} }{\|x\|^{2d_0}} 
P(x) = 
\sum_{k=1}^t E_{\mathcal{C}_k}^\tr 
\bigg[ \hat{S}_{0,k}\!\left(\frac{r x}{\|x\|} \right) + \sum_{j=1}^{q} \frac{r^{2d_j}}{\|x\|^{2d_j}}  \,g_j(x)\, \hat{S}_{j,k}\!\left(\frac{r x}{\|x\|} \right) \bigg]
E_{\mathcal{C}_k},
\end{equation}
where we have used the fact that $g_{q+1}(r x \|x\|^{-1}) = g_{q+2}(r x \|x\|^{-1})  = 0$.
Let $\omega$ be the smallest integer such that
\begin{equation*}
2\omega \geq 2d_0 + \sum_{j} 2d_j + \sum_{j,k} \deg(\hat{S}_{j,k})
\end{equation*}
and set $\nu := \omega - d_0$. Multiplying~\cref{e:augmented-putinar-decomposition-2} by $\|x\|^{2\omega}$ and rearranging, we obtain
\begin{equation}\label{e:augmented-putinar-decomposition-3}
\|x\|^{2\nu} 
P(x) = 
\sum_{k=1}^t E_{\mathcal{C}_k}^\tr 
\bigg( {S}_{0,k}\!\left(x\right) + \sum_{j=1}^{q}  g_j(x)\, {S}_{j,k}\!\left(x\right) \bigg)
E_{\mathcal{C}_k}
\end{equation}
with
\begin{align*}
S_{0,k}(x) &:= \frac{\|x\|^{2\omega}}{r^{2d_0}} \,\hat{S}_{0,k}\!\left(\frac{r x}{\|x\|} \right), &
S_{j,k}(x) &:= \frac{\|x\|^{2\omega - 2d_j}}{r^{2d_0-2d_j} }  \, \hat{S}_{j,k}\!\left(\frac{r x}{\|x\|} \right).
\end{align*}
\Cref{lemma:sos-symmetric-matrix} guarantees that these matrices are homogeneous and SOS. Since~\cref{e:augmented-putinar-decomposition-3} clearly holds also for $x = 0$,  it is the desired chordal SOS decomposition of $P$.

%% file: ss-conclusion.tex
We have proven SOS decomposition theorems for positive semidefinite polynomial matrices with chordal sparsity (\cref{th:weighted-chordal-decomposition,,th:sparse-reznick-homog,,th:sparse-reznick,,th:sparse-putinar-vasilescu-homog,,th:sparse-putinar-vasilescu,,th:sparse-putinar}), which can be viewed as sparsity-exploiting versions of the Hilbert--Artin, Reznick, Putinar, and Putinar--Vasilescu Positivstellens\"atze for polynomial matrices. Our theorems extend in a nontrivial way a classical chordal decomposition result for sparse numeric matrices~\cite{agler1988positive}, and we have shown that a na\"ive adaptation of this classical result to sparse polynomial matrices fails (\cref{th:failure-basic-decomposition}). 

In addition to being interesting in their own right, our SOS chordal decompositions have two important consequences. First, they can be combined with a straightforward scalarization argument to deduce new SOS representation results for nonnegative polynomials that are quadratic and correlatively sparse with respect to a subset of independent variables (\cref{corollary:globalcase,,corollary:global_even,corollary:compact}).  These statements specialize a sparse version of Putinar's Positivstellensatz proven in~\cite{lasserre2006convergent}, as well as recent sparsity-exploiting extensions of Reznick's Positivstellensatz~\cite{mai2020sparse}. 
Second, \cref{th:sparse-reznick-homog,,th:sparse-reznick,,th:sparse-putinar-vasilescu-homog,,th:sparse-putinar-vasilescu,,th:sparse-putinar} enable us to build new sparsity-exploiting hierarchies of SOS reformulations for convex optimization problems subject to large-scale but sparse polynomial matrix inequalities. These hierarchies are asymptotically exact for problems that have strictly feasible points and whose matrix inequalities are either imposed on a compact set satisfying the Archimedean condition (\cref{th:sos-program-convergence_putinar}), or satisfy additional homogeneity and strict positivity conditions (\cref{th:sos-program-convergence-homog,th:sos-program-convergence-inhomog}). Moreover, and perhaps most importantly, our SOS hierarchies have significantly lower computational complexity than traditional ones when the maximal cliques of the sparsity graph associated to the polynomial matrix inequality are much smaller than the matrix. As demonstrated by the numerical examples in \cref{s:examples}, this makes it possible to solve optimization problems with polynomial matrix inequalities that are well beyond the reach of standard SOS methods, without sacrificing their asymptotic convergence. 

It would be interesting to explore if the results we have presented in this work can be extended in various directions. For example, it may be possible to adapt the analysis in~\cite{scherer2006matrix} to derive a more general version of \cref{th:sparse-putinar}. It should also be possible to deduce explicit degree bounds for the SOS matrices that appear in all of our decomposition results. 
Stronger decomposition results for inhomogeneous polynomial matrix inequalities imposed on semialgebraic sets that are noncompact or do not satisfy the Archimedean condition would also be of interest. For instance, \cref{th:sparse-putinar-vasilescu,th:sparse-reznick} have restrictive assumptions on the behaviour of the leading homogeneous part of a polynomial matrix. These assumptions often are not met and, in such cases, SOS reformulations of convex optimization problems with polynomial matrix inequalities cannot be guaranteed to converge using \cref{th:sparse-putinar-vasilescu,th:sparse-reznick}. Finally, the chordal decomposition problem for semidefinite matrices has a dual formulation that considers positive semidefinite completion of partially specified matrices; see, e.g.,~\cite[Chapter 10]{vandenberghe2015chordal}. Building on a notion of SOS matrix completion introduced in~\cite{zheng2018decomposition}, it may be possible to establish SOS completion results for polynomial matrices.
All of these extensions will contribute to building a comprehensive theory for SOS decomposition and completion of polynomial matrices, which will enable the application of SOS programming to tackle large-scale optimization problems with semidefinite constraints on sparse polynomial matrices.

%% file: aa-psd-not-sos-example-proof.tex
\section{The matrix in \texorpdfstring{\cref{ex:nondecomposable-multivariate}}{Example \ref{ex:nondecomposable-multivariate}} is positive definite}\label{app:motzkin-pd}

For $x=0$, $P(0)=\frac{1}{100}\left[\begin{smallmatrix}101 & 0 & 0\\ 0 & 100 & 0\\ 0 & 0 & 100\end{smallmatrix} \right]$ is positive definite. For nonzero $x$, write
\begin{equation*}
P(x) = \begin{bmatrix} 0.01 & -0.01x_1 & 0 \\
-0.01x_1 & x_1^6+x_2^6+\tfrac12 & -x_2 \\
0& -x_2 & 1 \end{bmatrix}
+
\begin{bmatrix} 0.01(x_1^6+x_2^6)+ q(x) & 0 & 0 \\
0 & \tfrac12 & 0 \\
0& 0 & x_1^6 + x_2^6 \end{bmatrix}.
\end{equation*}
Since the second matrix on the right-hand side is positive definite, it suffices to show that the first one is PSD. This is true because its diagonal entries, its determinant, and its $2\times 2$ principal minors are nonnegative (confirmation of this is left to the reader).
%
%

%% file: aa-homogenization.tex
\section{Homogenization for \texorpdfstring{\cref{corollary:globalcase}}{Corollary~\ref{corollary:globalcase}}}
\label{app:homogenization}
If $p(x,y)$ is quadratic but not homogeneous with respect to $y$, introduce a new variable $z$ and define 
$$q(x,y,z) := z^2 p(x, z^{-1}y).$$
This polynomial is well defined when $z \neq 0$, can be extended by continuity to $z=0$, is both homogeneous and quadratic with respect to $(y,z)$, and satisfies $q(x,y,1)=p(x,y)$.

Since $z$ multiplies all entries of $y$, the correlative sparsity graph of $q$ with respect to $(y,z)$ is chordal and has maximal cliques
$\hat{\mathcal{C}}_1 = \mathcal{C}_1 \cup \{m+1\}$, 
$\hat{\mathcal{C}}_2 = \mathcal{C}_2 \cup \{m+1\}$, 
$\ldots$, 
$\hat{\mathcal{C}}_t = \mathcal{C}_t \cup \{m+1\}$, 
where $\mathcal{C}_1,\,\ldots,\,\mathcal{C}_t$ are the maximal cliques of the correlative sparsity graph of $p$ with respect to $y$. {Moreover, since both $p(x,y)$ and $\sum_{\alpha,|\beta|=2}c_{\alpha,\beta} x^\alpha y^\beta$ are nonnegative globally by assumption, $q(x,y,z)$ is nonnegative for all $x$, $y$ and $z$.} Applying the result of \cref{corollary:globalcase} for the homogeneous case to $q$, we find SOS polynomials $\hat{\sigma}_k(x,y_{\mathcal{C}_k},z)$, each homogeneous and quadratic in $y_{\mathcal{C}_k}$ and $z$, such that
\begin{equation*}
    \sigma_0(x) q(x,y,z) = \sum_{k=1}^t \hat{\sigma}_k(x,y_{\mathcal{C}_k},z).
\end{equation*}
Setting $z=1$ yields
\begin{equation*}
    \sigma_0(x) p(x,y) = \sum_{k=1}^t \hat{\sigma}_k(x,y_{\mathcal{C}_k},1),
\end{equation*}
which is the decomposition stated in \cref{corollary:globalcase} with polynomials $\sigma_k(x,y_{\mathcal{C}_k}) :=  \hat{\sigma}_k(x,y_{\mathcal{C}_k},1)$ that are quadratic (but not necessarily homogeneous) in $y_{\mathcal{C}_k}$.

%% file: aa-polynomial-diagonalization.tex
\section{Proof of \texorpdfstring{\cref{Prop:diagonalization}}{Proposition~\ref{Prop:diagonalization}}}
\label{app:diagonalization}


\Cref{Prop:diagonalization} is obvious if $m=1$, and follows directly from the next lemma if $m= 2$.
 
\begin{lemma}[Schm\"udgen~\cite{Schmudgen2009noncommutative}] \label{lemma:schmudegen}
    Let $P(x)$ be an $m \times m$ polynomial matrix with block form
    \begin{equation*}
        P(x) = \begin{bmatrix} u(x) & v(x)^\tr \\ v(x) & W(x) \end{bmatrix},
    \end{equation*}
    where $u$ is a polynomial, $v=\begin{bmatrix}v_1,\ldots,v_{m-1}\end{bmatrix}^\tr$ is a polynomial vector, and $W$ is a symmetric $(m-1)\times (m-1)$ polynomial matrix. Then, $u^4(x) P(x) = Z(x) Q(x) Z(x)^\tr$ with
    \begin{align*}
        Z(x) &= \begin{bmatrix} u(x) & 0 \\ v(x) & u(x) I_{r-1} \end{bmatrix},
        &
        Q(x) &= \begin{bmatrix} u^3(x) & 0 \\ 0 & u(x)^2 W(x) - u(x) v(x) v(x)^\tr \end{bmatrix}.
    \end{align*}
\end{lemma}

For $m\geq 3$, we use an induction procedure that combines Schm\"udgen's lemma with the zero fill-in property of the Cholesky algorithm for matrices with chordal sparsity.

Assume that \cref{Prop:diagonalization} holds for all polynomial matrices of size $m-1$ with chordal sparsity. We claim that it holds also for polynomial matrices of size $m$. Let $P(x)$ be any $m \times m$ matrix whose sparsity graph $\mathcal{G}$ is chordal.
By Theorem~3.3 in~\cite{vandenberghe2015chordal}, the graph $\mathcal{G}$ has at least one simplicial vertex. 
Let $\Pi$ be a permutation matrix and denote by $\mathcal{G}_\Pi$ the sparsity graph of the permuted matrix $\Pi P \Pi^\tr$, which is obtained simply by reordering the vertices of $\mathcal{G}$ as specified by the permutation $\Pi$. We choose $\Pi$ such that vertex $1$ is simplicial for $\mathcal{G}_\Pi$ and such that the first maximal clique of $\mathcal{G}_\Pi$ is $\mathcal{C}_1 = \{1,\ldots,r\}$ for some $r \geq 1$. This means that the matrix $\Pi P \Pi^\tr$ can be partitioned into the block form
 \begin{equation} \label{Eq:Induction_1}
      \Pi P(x) \Pi^\tr = \begin{bmatrix}
            a(x) & q(x)^\tr & 0 \\
            q(x) & F(x)     & G(x)^\tr \\
            0    & G(x)     & H(x) \\
        \end{bmatrix},
\end{equation}
where $a$ is a polynomial, $q=\begin{bmatrix}q_1,\ldots,q_{r-1}\end{bmatrix}^\tr$ is a vector of polynomials, and $F$, $G$ and $H$ are polynomial matrices of suitable dimensions.
    
Applying \cref{lemma:schmudegen} with 
$$
u=a, \quad v = \begin{bmatrix}q\\0\end{bmatrix}, \quad W = \begin{bmatrix}F&G^\tr\\G&H\end{bmatrix}
$$ 
yields (omitting the argument $x$ from all polynomials to ease the notation)
\begin{equation} \label{Eq:Induction_s1a}
    a^4 \Pi P \Pi^\tr = 
    \underbrace{\begin{bmatrix} a & 0 & 0\\q & a I & 0 \\ 0 & 0 &  a I\end{bmatrix}}_{:=Z}
    \underbrace{\begin{bmatrix} a^3 & 0 & 0 \\ 0 &a^2 F  - aqq^{\tr} & a^2G^\tr \\0 & a^2 G & a^2H\end{bmatrix}}_{:=Q}
    \underbrace{\begin{bmatrix} a & q^\tr & 0\\0 & a I & 0 \\ 0 & 0 &  a I\end{bmatrix}}_{:=Z^\tr}.
\end{equation}

Next, consider the matrix
\begin{equation*}
    P'(x) = \begin{bmatrix}a^2 F  - aqq^{\tr} & a^2G^\tr \\a^2 G & a^2H\end{bmatrix}.
\end{equation*}
The sparsity graph $\mathcal{G}'$ of $P'$ coincides with the subgraph of $\mathcal{G}$ obtained by removing vertex~$1$. Since this vertex is simplicial and $\mathcal{G}$ is chordal, $\mathcal{G}'$ is also chordal~\cite[Section 4.2]{vandenberghe2015chordal}. Thus, $P'$ is an $(m-1)\times (m-1)$ matrix with a chordal sparsity graph. 
By our induction assumption, there exists 
an $(m-1)\times(m-1)$ permutation matrix $\Lambda$,
an $(m-1)\times(m-1)$ lower-triangular polynomial matrix $R$, a polynomial $s$, and polynomials $d_2,\ldots,d_{m}$ such that
\begin{equation*}
    s^4 \Lambda P' \Lambda^\tr = R \Diag(d_2,\ldots,d_{m}) R^\tr.
\end{equation*}
Moreover, $R + R^\tr$ has the same sparsity pattern as $\Lambda P' \Lambda^\tr$,  meaning that $\Lambda^\tr (R + R^\tr) \Lambda$ has the same sparsity as $P'$.
Combining this factorization with~\cref{Eq:Induction_s1a} we obtain
\begin{align}
    s^4 a^4 \Pi P \Pi^\tr &= Z  \begin{bmatrix} s^4 a^3 & 0  \\ 0 & s^4 P' \end{bmatrix} Z^\tr 
              \nonumber\\
              &= Z \begin{bmatrix} s^4 a^3 & 0  \\ 0 & \Lambda^\tr R \Diag(d_2,\,\ldots,\,d_m) R^\tr \Lambda \end{bmatrix} Z^\tr 
              \nonumber\\
              &= Z
                    \begin{bmatrix} 1 & 0 \\0& \Lambda^\tr \end{bmatrix} 
                    \begin{bmatrix} 1 & 0 \\0& R \end{bmatrix} 
                    \Diag( s^4 a^3,\,d_2,\,\ldots,\,d_m) 
                    \begin{bmatrix} 1 & 0 \\0& R^\tr \end{bmatrix}
                    \begin{bmatrix} 1 & 0 \\0& \Lambda \end{bmatrix} 
                 Z^\tr.
              \label{Eq:Induction_s3}
\end{align}

To conclude the proof of \cref{Prop:diagonalization}, set $b := sa$, $d_1:= s^4 a^3$ and define
\begin{align*}
    T := \begin{bmatrix} 1 & 0 \\0& \Lambda \end{bmatrix} \Pi,
    \qquad
    L :=
    \begin{bmatrix} 1 & 0 \\0& \Lambda \end{bmatrix} 
    Z 
    \begin{bmatrix} 1 & 0 \\0& \Lambda^\tr \end{bmatrix} 
    \begin{bmatrix} 1 & 0 \\0& R \end{bmatrix}.
\end{align*}
Note that $T$ is a permutation matrix, while $L$ is lower triangular.
Pre- and post-multiplying identity~\cref{Eq:Induction_s3} by $\left[\begin{smallmatrix} 1 & 0 \\0& \Lambda \end{smallmatrix}\right]$ and $\left[\begin{smallmatrix} 1 & 0 \\0& \Lambda^\tr \end{smallmatrix}\right]$, respectively, gives
$$
b^4 T P T^\tr = L \Diag(d_1,\ldots,d_m) L^\tr,
$$
which is the desired factorization.
It remains to verify that $L + L^\tr$ has the same sparsity pattern as $T P T^\tr$ or, equivalently, that $T^\tr(L+L^\tr)T$ has the same sparsity pattern as $P$. To see this, write $Z=\left[\begin{smallmatrix}a & 0\\ v & aI\end{smallmatrix}\right]$ with $v = \left[\begin{smallmatrix}q\\ 0\end{smallmatrix}\right]$ and observe that 
\begin{equation*}
\begin{aligned}
    T^\tr (L + L^\tr) T &= \Pi^\tr \begin{bmatrix} 1 & 0 \\0& \Lambda^\tr  \end{bmatrix} \left(\begin{bmatrix} 1 & 0 \\0& \Lambda \end{bmatrix} 
    Z 
    \begin{bmatrix} 1 & 0 \\0& \Lambda^\tr \end{bmatrix} 
    \begin{bmatrix} 1 & 0 \\0& R \end{bmatrix} + L^\tr \right) \begin{bmatrix} 1 & 0 \\0& \Lambda \end{bmatrix} \Pi \\
    & =  \Pi^\tr\left( 
        \begin{bmatrix}a & 0\\ v & aI\end{bmatrix}  
        \begin{bmatrix} 1 & 0 \\0& \Lambda^\tr R \Lambda \end{bmatrix}
        + 
        \begin{bmatrix} 1 & 0 \\0& \Lambda^\tr R^\tr \Lambda \end{bmatrix}
        \begin{bmatrix}a & v^\tr\\ 0 & aI\end{bmatrix}
        \right) \Pi
        \\
    & = \Pi^\tr
        \begin{bmatrix}a & v^\tr\\ v & a \Lambda^\tr (R+R^\tr) \Lambda\end{bmatrix}  
        \Pi.
\end{aligned}
\end{equation*}
Since $\Lambda^\tr (R+R^\tr) \Lambda$ has the same sparsity pattern as $P'$ and $v = \left[\begin{smallmatrix}q\\ 0\end{smallmatrix}\right]$, the $2\times 2$ block matrix on the right-hand side has the same sparsity pattern as the right-hand side of~\cref{Eq:Induction_1}, hence as $\Pi P \Pi^\tr$. We conclude that $T^\tr (L + L^\tr) T$ has the same sparsity pattern as $P$, as required.

%% file: aa-weighted-sos-matrix-lemma.tex
\section{Proof of \texorpdfstring{\cref{lemma:sos-symmetric-matrix}}{Lemma~\ref{lemma:sos-symmetric-matrix}}}
\label{app:sos-symmetric-matrix-lemma}

It suffices to consider $2\omega = \deg(S)$. Since $S$ is SOS and $S(-x)=S(x)$, symmetry arguments~\cite{lofberg2009pre} imply that $S(x)= S_1(x) + S_2(x)$ with $S_1(x) = H_e(x)^\tr H_e(x)$, $S_2(x) = H_o(x)^\tr H_o(x)$, and
\begin{align*}
H_e(x) &= \sum_{\substack{\alpha \in \mathbb{N}^n\\\abs{\alpha} \leq \omega\text{ \& even}}} \!\! A_\alpha \, x^\alpha,
&
H_o(x) &= \sum_{\substack{\alpha \in \mathbb{N}^n\\\abs{\alpha} \leq \omega\text{ \& odd}}} \!\! B_\alpha \, x^\alpha,
\end{align*}
where $A_\alpha$ and $B_\alpha$ are $m \times m$ coefficient matrices. Therefore, we only need to show that $\|x\|^{2\omega} S_1(rx \|x\|^{-1})$ and $\|x\|^{2\omega} S_2(rx \|x\|^{-1})$ are SOS and homogeneous of degree $2\omega$. If $r=0$, this is trivial. If $r>0$, set $A_\alpha':= r^{\abs{\alpha}}A_\alpha$ and write
\begin{equation}\label{e:even-matrix-factorization}
\|x\|^{2\omega} S_1\!\left(\frac{rx}{\|x\|}\right) = 
\left( \sum_{\abs{\alpha} \leq \omega, \text{even}} \!\!\!\! A_\alpha' \, x^\alpha \|x\|^{\omega - \abs{\alpha}} \right)^\tr\!
\left( \sum_{\abs{\alpha} \leq \omega, \text{even}} \!\!\!\! A_\alpha' \, x^\alpha \|x\|^{\omega - \abs{\alpha}} \right).
\end{equation}
To show that this matrix is SOS, we distinguish two cases. If $\omega$ is even, then so is $\omega-\abs{\alpha}$, and $\|x\|^{\omega-\abs{\alpha}}$ is a polynomial of $x$. In this case, each term in brackets on the right-hand side of~\cref{e:even-matrix-factorization} is a polynomial matrix, so $\|x\|^{2\omega} S_1(x \|x\|^{-1})$ is SOS. If $\omega$ is odd, instead,~\cref{e:even-matrix-factorization} can be written as
\begin{equation*}
\|x\|^{2\omega} S_1\!\left(\frac{rx}{\|x\|}\right) = 
\|x\|^2\left( \sum_{\abs{\alpha} \leq \omega, \text{even}} \!\!\!\! A_\alpha' \, x^\alpha \|x\|^{\omega - \abs{\alpha}-1} \right)^\tr\!
\left( \sum_{\abs{\alpha} \leq \omega, \text{even}} \!\!\!\! A_\alpha' \, x^\alpha \|x\|^{\omega - \abs{\alpha}-1} \right).
\end{equation*}
Since $\omega - \abs{\alpha}-1$ is even, the right-hand side is an SOS polynomial matrix, so $\|x\|^{2\omega} S_1(rx \|x\|^{-1})$ is SOS. In both cases, the matrix in the right-hand side of~\cref{e:even-matrix-factorization} is clearly homogeneous of degree $2\omega$, and so is the left-hand side. Analogous reasoning proves that $\|x\|^{2\omega} S_2(rx \|x\|^{-1})$ is SOS and homogeneous with degree $2\omega$, concluding the proof.

%% file: main_arXiv.bbl
\begin{thebibliography}{10}

\bibitem{chesi2010lmi}
G.~Chesi.
\newblock {LMI techniques for optimization over polynomials in control: a
  survey}.
\newblock {\em IEEE Trans. Automat. Control}, 55(11):2500--2510, 2010.

\bibitem{lasserre2010moments}
J.-B. Lasserre.
\newblock {\em {Moments, Positive Polynomials and their Applications}}.
\newblock Imperial College Press, 2010.

\bibitem{henrion2011inner}
Didier Henrion and Jean-Bernard Lasserre.
\newblock Inner approximations for polynomial matrix inequalities and robust
  stability regions.
\newblock {\em IEEE Trans. Automat. Control}, 57(6):1456--1467, 2011.

\bibitem{scherer2006lmi}
Carsten~W Scherer.
\newblock {LMI relaxations in robust control}.
\newblock {\em Eur. J. Control}, 12(1):3--29, 2006.

\bibitem{Murty1987}
K.~G. Murty and S.~N. Kabadi.
\newblock {Some NP-complete problems in quadratic and nonlinear programming}.
\newblock {\em Math. Program.}, 39(2):117--129, 1987.

\bibitem{gatermann2004symmetry}
K.~Gatermann and P.~A. Parrilo.
\newblock Symmetry groups, semidefinite programs, and sums of squares.
\newblock {\em J. Pure Appl. Algebra}, 192(1-3):95--128, 2004.

\bibitem{kojima2003sums}
M~Kojima.
\newblock Sums of squares relaxations of polynomial semidefinite programs.
\newblock Research Reports on Mathematical and Computing Sciences Series B:
  Operations Research B-397, Tokyo Institute of Technology, 2003.

\bibitem{parrilo2013semidefinite}
P.~A. Parrilo.
\newblock {Polynomial optimization , sums of squares and applications}.
\newblock In G.~Blekherman, P.~A.. Parrilo, and R.~R.. Thomas, editors, {\em
  Semidefinite optimization and convex algebraic geometry}, chapter~3, pages
  47--157. SIAM, 1st edition, 2013.

\bibitem{scherer2006matrix}
C.~Scherer and C.~Hol.
\newblock Matrix sum-of-squares relaxations for robust semi-definite programs.
\newblock {\em Math. Program.}, 107:189--211, 2006.

\bibitem{boyd2004convex}
S.~Boyd and L.~Vandenberghe.
\newblock {\em Convex optimization}.
\newblock Cambridge University Press, 2004.

\bibitem{nemirovski2006advances}
A.~Nemirovski.
\newblock {Advances in convex optimization: Conic programming}.
\newblock In {\em International Congress of Mathematicians}, volume~1, pages
  413--444, 2006.

\bibitem{nesterov1994interior}
Y.~Nesterov and A.~Nemirovski.
\newblock {\em {Interior-Point Polynomial Algorithms in Convex Programming}}.
\newblock SIAM, 1994.

\bibitem{vandenberghe1996semidefinite}
L.~Vandenberghe and S.~Boyd.
\newblock {Semidefinite Programming}.
\newblock {\em SIAM Rev.}, 38(1):49--95, 1996.

\bibitem{fukuda2001exploiting}
M.~Fukuda, M.~Kojima, K.~Murota, and K.~Nakata.
\newblock {Exploiting sparsity in semidefinite programming via matrix
  completion I: General framework}.
\newblock {\em SIAM J. Optim.}, 11(3):647--674, 2001.

\bibitem{nakata2003exploiting}
K.~Nakata, K.~Fujisawa, M.~Fukuda, M.~Kojima, and K.~Murota.
\newblock {Exploiting sparsity in semidefinite programming via matrix
  completion II: Implementation and numerical results}.
\newblock {\em Math. Program. B}, 95(2):303--327, 2003.

\bibitem{sun2014decomposition}
Y.~Sun, M.~S. Andersen, and L.~Vandenberghe.
\newblock Decomposition in conic optimization with partially separable
  structure.
\newblock {\em SIAM J. Optim.}, 24(2):873--897, 2014.

\bibitem{vandenberghe2015chordal}
L.~Vandenberghe, M.~S. Andersen, et~al.
\newblock Chordal graphs and semidefinite optimization.
\newblock {\em Found. Trends Optim.}, 1(4):241--433, 2015.

\bibitem{zheng2020chordal}
Y.~Zheng, G.~Fantuzzi, A.~Papachristodoulou, P.~Goulart, and A.~Wynn.
\newblock {Chordal decomposition in operator-splitting methods for sparse
  semidefinite programs}.
\newblock {\em Math. Program.}, 180:489--532, 2020.

\bibitem{putinar1993positive}
M.~Putinar.
\newblock Positive polynomials on compact semi-algebraic sets.
\newblock {\em Indiana Univ. Math. J.}, 42(3):969--984, 1993.

\bibitem{PutinarVasilescu1999}
M.~Putinar and F.-H. Vasilescu.
\newblock {Positive polynomials on semi-algebraic sets}.
\newblock {\em C. R. Math. Acad. Sci. Paris}, 328(7):585--589, 1999.

\bibitem{Reznick1995}
B~Reznick.
\newblock {Uniform denominators in Hilbert's seventeenth problem}.
\newblock {\em Math. Z.}, 220:75--97, 1995.

\bibitem{artin1927zerlegung}
E.~Artin.
\newblock {{\"U}ber die Zerlegung definiter Funktionen in Quadrate}.
\newblock {\em Abh. Math. Semin. Univ. Hambg.}, 5(1):100--115, 1927.

\bibitem{Du2017}
T.~H.-B. Du.
\newblock {A note on Positivstellens\"atze for matrix polynomials}.
\newblock {\em East-West J. Math.}, 19(2):171--182, 2017.

\bibitem{Dinh2021}
T.~H. Dinh, M.~T. Ho, and C.~T. Le.
\newblock {Positivstellens{\"{a}}tze for polynomial matrices}.
\newblock {\em Positivity}, 2021.

\bibitem{agler1988positive}
J.~Agler, W.~Helton, S.~McCullough, and L.~Rodman.
\newblock Positive semidefinite matrices with a given sparsity pattern.
\newblock {\em Linear Algebra Appl.}, 107:101--149, 1988.

\bibitem{andersen2014robust}
M.~S. Andersen, S.~K. Pakazad, A.~Hansson, and A.~Rantzer.
\newblock Robust stability analysis of sparsely interconnected uncertain
  systems.
\newblock {\em IEEE Trans. Automat. Control}, 59(8):2151--2156, 2014.

\bibitem{Zheng2017Scalable}
Y.~Zheng, R.~P Mason, and A.~Papachristodoulou.
\newblock Scalable design of structured controllers using chordal
  decomposition.
\newblock {\em IEEE Trans. Automat. Control}, 63(3):752--767, 2018.

\bibitem{andersen2014reduced}
M.~S. Andersen, A.~Hansson, and L.~Vandenberghe.
\newblock Reduced-complexity semidefinite relaxations of optimal power flow
  problems.
\newblock {\em IEEE Trans. Power Syst.}, 29(4):1855--1863, 2014.

\bibitem{molzahn2013implementation}
D.~K. Molzahn, J.~T. Holzer, B.~C. Lesieutre, and C.~L. DeMarco.
\newblock Implementation of a large-scale optimal power flow solver based on
  semidefinite programming.
\newblock {\em IEEE Trans. Power Syst.}, 28(4):3987--3998, 2013.

\bibitem{reznick1978extremal}
B.~Reznick.
\newblock {Extremal PSD forms with few terms}.
\newblock {\em Duke Math. J.}, 45(2):363--374, 1978.

\bibitem{permenter2014basis}
F.~Permenter and P.~A. Parrilo.
\newblock {Basis selection for SOS programs via facial reduction and polyhedral
  approximations}.
\newblock In {\em Proceedings of the 53\textsuperscript{rd} IEEE Conference on
  Decision and Control}, pages 6615--6620, 2014.

\bibitem{lofberg2009pre}
J.~L\"{o}fberg.
\newblock Pre-and post-processing sum-of-squares programs in practice.
\newblock {\em IEEE Trans. Automat. Control}, 54(5):1007--1011, 2009.

\bibitem{Riener2013}
C.~Riener, T.~Theobald, L.~J. Andr{\'{e}}n, and J.~B. Lasserre.
\newblock {Exploiting symmetries in SDP-relaxations for polynomial
  optimization}.
\newblock {\em Math. Oper. Res.}, 38(1):122--141, 2013.

\bibitem{waki2006sums}
H.~Waki, S.~Kim, M.~Kojima, and M.~Muramatsu.
\newblock Sums of squares and semidefinite program relaxations for polynomial
  optimization problems with structured sparsity.
\newblock {\em SIAM J. Optim.}, 17(1):218--242, 2006.

\bibitem{lasserre2006convergent}
J.-B. Lasserre.
\newblock {Convergent SDP-relaxations in polynomial optimization with
  sparsity}.
\newblock {\em SIAM J. Optim.}, 17(3):822--843, 2006.

\bibitem{grimm2007note}
D.~Grimm, T.~Netzer, and M.~Schweighofer.
\newblock A note on the representation of positive polynomials with structured
  sparsity.
\newblock {\em Arch. Math. (Basel)}, 89(5):399--403, 2007.

\bibitem{klep2019sparse}
I.~Klep, V.~Magron, and J.~Povh.
\newblock {Sparse noncommutative polynomial optimization}.
\newblock {\em Math. Program.}, 01:1--37, 2021.

\bibitem{josz2018lasserre}
C.~Josz and D.~K. Molzahn.
\newblock Lasserre hierarchy for large scale polynomial optimization in real
  and complex variables.
\newblock {\em SIAM J. Optim.}, 28(2):1017--1048, 2018.

\bibitem{Wang2019term-sparsity}
J.~Wang, H.~Li, and B.~Xia.
\newblock {A new sparse SOS decomposition algorithm based on term sparsity}.
\newblock {\em Proceedings of the International Symposium on Symbolic and
  Algebraic Computation, ISSAC}, pages 347--354, 2019.

\bibitem{Wang2020chordal-tssos}
J.~Wang, V.~Magron, and J.-B. Lasserre.
\newblock {Chordal-TSSOS: a moment-SOS hierarchy that exploits term sparsity
  with chordal extension}.
\newblock {\em SIAM J. Optim.}, 31(1):114--141, 2021.

\bibitem{Wang2020tssos}
J.~Wang, V.~Magron, and J.-B. Lasserre.
\newblock {TSSOS: A moment-SOS hierarchy that exploits term sparsity}.
\newblock {\em SIAM J. Optim.}, 31(1):30--58, 2021.

\bibitem{Wang2020cs-tssos}
J.~Wang, V.~Magron, J.~B Lasserre, and N.~H.~A. Mai.
\newblock {CS-TSSOS: Correlative and term sparsity for large-scale polynomial
  optimization}.
\newblock \href{http://arxiv.org/abs/2005.02828}{arXiv:2005.02828} [math.OC],
  2020.

\bibitem{zheng2018sparse}
Y.~Zheng, G.~Fantuzzi, and A.~Papachristodoulou.
\newblock {Sparse sum-of-squares (SOS) optimization: A bridge between
  DSOS/SDSOS and SOS optimization for sparse polynomials}.
\newblock In {\em Proceedings of the 2019 American Control Conference}, pages
  5513--5518, 2019.

\bibitem{kakimura2010direct}
N.~Kakimura.
\newblock A direct proof for the matrix decomposition of chordal-structured
  positive semidefinite matrices.
\newblock {\em Linear Algebra Appl.}, 433(4):819--823, 2010.

\bibitem{Schmudgen2009noncommutative}
K.~Schm{\"u}dgen.
\newblock Noncommutative real algebraic geometry some basic concepts and first
  ideas.
\newblock In {\em Emerging Applications of Algebraic Geometry}, pages 325--350.
  Springer, 2009.

\bibitem{aylward2007explicit}
E.~M. Aylward, S.~M. Itani, and P.~A. Parrilo.
\newblock {Explicit {SOS} decompositions of univariate polynomial matrices and
  the Kalman--Yakubovich--Popov lemma}.
\newblock In {\em Proceedings of the 46\textsuperscript{th} IEEE Conference on
  Decision and Control}, pages 5660--5665, 2007.

\bibitem{Motzkin1967}
T.~S. Motzkin.
\newblock {The arithmetic-geometric inequality}.
\newblock In {\em Inequalities (Proc. Sympos. Wright-Patterson Air Force Base,
  Ohio, 1965)}, pages 205--224, 1967.

\bibitem{Laurent2009}
M.~Laurent.
\newblock {Sums of Squares, Moment Matrices and Optimization Over Polynomials}.
\newblock In M.~Putinar and S.~Sullivant, editors, {\em Emerging Applications
  of Algebraic Geometry}, The IMA Volumes in Mathematics and its Applications,
  pages 157--270. Springer, New York, NY, 2009.

\bibitem{lasserre2015introduction_book}
J.-B. Lasserre.
\newblock {\em {An introduction to polynomial and semi-algebraic
  optimization}}.
\newblock Cambridge University Press, 2015.

\bibitem{Nie2008}
J.~Nie and J.~Demmel.
\newblock {Sparse SOS relaxations for minimizing functions that are summations
  of small polynomials}.
\newblock {\em SIAM J. Optim.}, 19(4):1534--1558, 2008.

\bibitem{mai2020sparse}
N.~H.~A. Mai, V.~Magron, and J.-B. Lasserre.
\newblock {A sparse version of Reznick's Positivstellensatz}.
\newblock \href{https://arXiv.org/abs/2002.05101}{arXiv:2002.05101} [math.OC],
  2020.

\bibitem{andersen2000mosek}
E.~D. Andersen and K.~D. Andersen.
\newblock The {MOSEK} interior point optimizer for linear programming: an
  implementation of the homogeneous algorithm.
\newblock In {\em High performance optimization}, pages 197--232. Springer,
  2000.

\bibitem{lofberg2004yalmip}
J.~L{\"o}fberg.
\newblock {YALMIP: A toolbox for modeling and optimization in MATLAB}.
\newblock In {\em Proceedings of the IEEE International Symposium on
  Computer-Aided Control System Design}, pages 284--289, 2004.

\bibitem{mason2015chordal}
Richard Mason.
\newblock {\em A chordal sparsity approach to scalable linear and nonlinear
  systems analysis}.
\newblock PhD thesis, University of Oxford, 2015.

\bibitem{rose1970triangulated}
D.~J. Rose.
\newblock Triangulated graphs and the elimination process.
\newblock {\em J. Math. Anal. Appl}, 32(3):597--609, 1970.

\bibitem{zheng2018decomposition}
Y.~Zheng, G.~Fantuzzi, and A.~Papachristodoulou.
\newblock Decomposition and completion of sum-of-squares matrices.
\newblock In {\em Proceedings of the 57\textsuperscript{th} IEEE Conference on
  Decision and Control}, pages 4026--4031, 2018.

\end{thebibliography}
